\documentclass[12pt,a4paper, oneside]{amsart}

\usepackage{amssymb,amsmath,amsthm}
\usepackage{  verbatim, amsfonts, mathtools, upgreek, xcolor, bbm}
\usepackage{graphics, xspace, enumerate}
\usepackage{stix}
\usepackage{dsfont}
\usepackage[a4paper,margin=2.5cm]{geometry}

\usepackage{graphicx}
\usepackage[colorlinks=true,citecolor=red,urlcolor=blue,linkcolor=red,bookmarksopen=true,unicode=true,pdffitwindow=true]{hyperref}
\usepackage[english]{babel}
\usepackage[languagenames,fixlanguage]{babelbib}
\hypersetup{pdfauthor={}}
\hypersetup{pdftitle={}}

\usepackage{seqsplit,cleveref}

\theoremstyle{plain}
\newtheorem{theorem}{Theorem}[section]

\newtheorem{lemma}[theorem]{Lemma}
\newtheorem{proposition}[theorem]{Proposition}
\theoremstyle{definition}
\newtheorem{remark}[theorem]{Remark}
\newtheorem{example}[theorem]{Example}

\newcommand {\Prob} {\ensuremath{\mathbb{P}}}
\newcommand {\R} {\ensuremath{\mathbb{R}}}

\newcommand {\N} {\ensuremath{\mathbb{N}}}

\newcommand{\df}{\coloneqq}

\newcommand{\bb}{\mathrm{b}}

\newcommand{\E}{\mathrm{e}}
\newcommand{\Id}{\mathbb{I}}

\newcommand{\X}{\mathsf{X}}
\newcommand{\B}{\mathsf{B}}
\newcommand{\LL}{\mathsf{L}}
\newcommand{\sub}{\mathsf{S}}
\newcommand{\p}{\mathsf{p}}

\newcommand{\D}{\mathrm{d}}

\numberwithin{equation}{section}

\title[On subgeometric ergodicity of regime-switching diffusion processes]{On subgeometric ergodicity of regime-switching diffusion processes}

\author[P.\ Lazi\'c]{Petra\ Lazi\'c}
\address[Petra\  Lazi\'c]{
	Department of Mathematics\\University of Zagreb\\ Zagreb\\Croatia}
\email{petralaz@math.hr}

\author[N.\ Sandri\'{c}]{Nikola Sandri\'{c}}
\address[Nikola\ Sandri\'{c}]{Department of Mathematics\\University of Zagreb\\ Zagreb\\Croatia}
\email{nsandric@math.hr}


\begin{document}
\allowdisplaybreaks[4]

\begin{abstract}
	 In this article, we discuss subgeometric ergodicity  of a class of regime-switching diffusion processes. We derive conditions on the drift and diffusion coefficients, and the switching mechanism which result in subgeometric ergodicity of the corresponding semigroup with respect to the total variation distance as well as a class of Wasserstein distances. At the end, subgeometric ergodicity of certain classes of regime-switching Markov processes with jumps is also discussed.
\end{abstract}

\maketitle

\section{Introduction}\label{S1} 
One of the classical directions in the analysis of regime switching systems centers around their ergodicity properties.
In this article, we  discuss subgeometric ergodicity of a regime-switching diffusion process $\{(\X(x,i;t),\Lambda(x,i;t))\}_{t\ge0}$
with respect to the total variation distance and/or a class of Wasserstein distances.
The first (continuous-state) component is
given by 
\begin{equation}
\begin{aligned}
\label{eq1}
\D \X(x,i;t) &\,=\, \bb\bigl(\X(x,i;t),\Lambda(x,i;t)\bigr)\D t+\upsigma\bigl(X(x,i;t),\Lambda(x,i;t)\bigr)\D  \B(t)\\ \X(x,i;0)&\,=\, x \in\R^d\\
\Lambda(x,i;0)&\,=\, i\in\mathbb{S}\,,
\end{aligned}
\end{equation} where $\{\B(t)\}_{t\ge0}$  denotes a standard $n$-dimensional Brownian motion (starting from the origin),  and  the second (regime-switching) component is a right-continuous temporally-homogeneous Markov chain  with finite state space $\mathbb{S}$.
The processes  $\{\B(t)\}_{t\ge0}$ and $\{\Lambda(x,i;t)\}_{t\ge0}$ are  both defined on a stochastic basis $(\Omega, \mathcal{F}, \{\mathcal{F}_t\}_{t\ge0},\Prob)$ satisfying the usual conditions. 
We assume that the coefficients $\bb:\R^d\times\mathbb{S}\to\R^d$ and $\upsigma:\R^d\times\mathbb{S}\to\R^{d\times n}$, and the process $\{\Lambda(x,i;t)\}_{t\ge0}$ satisfy the following:

\medskip

\begin{description}
		\item[(A1)]   for any $r>0$ and $i\in\mathbb{S}$, $$\sup_{x\in\mathscr{B}_r(0)}\bigl(|\bb(x,i)|+\lVert \upsigma(x,i)\lVert_{{\rm HS}}\bigr)\,<\,\infty\,,$$ where $\lVert \cdot\lVert_{{\rm HS}}$ denotes the  Hilbert-Schmidt norm (see below for the definition)

	\medskip

	\item[(A2)] for each $(x,i)\in\R^d\times\mathbb{S}$ the regime-switching stochastic differential equation (RSSDE) in  \cref{eq1} admits a unique nonexplosive strong solution $\{X(x,i;t)\}_{t\ge0}$ which has continuous sample paths

	\medskip
	
	\item[(A3)]	the process $\{(\mathsf{X}(x,i;t),\Lambda(x,i;t))\}_{t\ge0}$ is a temporally-homogeneous strong Markov process with transition kernel $\p(t,(x,i),\D y\times \{j\})=\Prob((\X(x,i;t),\Lambda(x,i;t))\in\D y\times \{j\})$
	
	\medskip
	
 \item[(A4)]	the corresponding semigroup of linear operators $\{\mathcal{P}_t\}_{t\ge0}$, defined by
 $$\mathcal{P}_tf(x,i)\,\df\, \int_{\R^d\times\mathbb{S}}f(y,j)\, \mathsf{p}\bigl(t,(x,i),\D y\times \{j\}\bigr)\,,\qquad   f\in \mathcal{B}_b(\R^d\times\mathbb{S})\,,$$ satisfies the 
		$\mathcal{C}_b$-Feller property, that is, $\mathcal{P}_t(\mathcal{C}_b(\R^d\times\mathbb{S}))\subseteq \mathcal{C}_b(\R^d\times\mathbb{S})$ for all $t\ge0$
		
		\medskip
		
		\item[(A5)]	for any $(x,i)\in\R^d\times\mathbb{S}$ and $f\in \mathcal{C}^2(\R^d\times\mathbb{S})$ the process $$\left\{f\bigl(\X(x,i;t),\Lambda(x,i;t)\bigr)-f(x,i)-\int_0^t\mathcal{L}f\bigl(\X(x,i;s),\Lambda(x,i;s)\bigr)\D s\right\}_{t\ge0}$$ is a $\mathbb{P}$-local martingale, where
		$$\mathcal{L}f(x,i)\,=\,\mathcal{L}_if(x,i)+\mathcal{Q}(x)f(x,i)$$ with
		$$\mathcal{L}_if(x)\,=\,\bigl\langle \bb(x,i),\nabla f(x)\bigr \rangle+\frac{1}{2}{\rm Tr}\bigl(\upsigma(x,i)\upsigma(x,i)^{T}\nabla^2f(x)\bigr)\,,\qquad f\in\mathcal{C}^2(\R^d)\,,$$ and $\mathcal{Q}(x)=(\mathrm{q}_{ij}(x))_{i,j\in\mathbb{S}}$ being the infinitesimal generator of the process $\{\Lambda(x,i;t)\}_{t\ge0}$, that is, 
		$$\mathcal{Q}(x)f(i)=\sum_{j\in\mathbb{S}}f(j)\, \mathrm{q}_{ij}(x)\,, \qquad f\in\mathbb{S}^\mathbb{S}\,,$$ with
		$\mathrm{q}_{ij}(x)=\lim_{t\to 0}\Prob(\Lambda(x,i;t)=j)/t$ for $i\neq j$ and $\mathrm{q}_{ii}(x)=-\sum_{j\neq i}\mathrm{q}_{ij}(x)$.
			\end{description}
		
		\medskip

	\noindent	
		Here, $\mathscr{B}_r(x)$  denotes the open ball with radius $r>0$ around $x\in\R^d$,  $\langle\cdot,\cdot \rangle$ stands for the standard scalar product on $\R^d$, $\lvert\cdot\rvert\df \langle\cdot,\cdot\rangle^{1/2}$ is the corresponding Euclidean norm and  $\lVert M\lVert_{{\rm HS}}^2:={\rm Tr}\,MM^T$ denotes the  Hilbert-Schmidt norm of a real  matrix $M.$
The symbols $\mathcal{B}(\R^d\times\mathbb{S})$,  $\mathcal{B}_b(\R^d\times\mathbb{S})$, $\mathcal{C}_b(\R^d\times\mathbb{S})$ and $\mathcal{C}^2(\R^d\times\mathbb{S})$  stand for the spaces of all functions $f:\R^d\times\mathbb{S}\to\R$ such that $x\mapsto f(x,i)$  is  Borel measurable, bounded and Borel measurable, bounded and continuous  and of class $\mathcal{C}^2$ for all $i\in\mathbb{S}$, respectively. We refer the readers to    \cite{Xi-Yin-Zhu-2019} (see also \cite{Kunwai-Zhu-2020} and \cite{Mao-Yuan-Book-2006}) for conditions  ensuring   \textbf{(A1)}-\textbf{(A5)}.

\subsection{Ergodicity of $\{(\X(x,i;t),\Lambda(x,i;t))\}_{t\ge0}$}

We recall some definitions and results from the ergodic theory of Markov processes. Our main references are \cite{Meyn-Tweedie-AdvAP-II-1993} and \cite{Tweedie-1994}.
The process $\{(\X(x,i;t),$ $\Lambda(x,i;t))\}_{t\ge0}$ is said to be

\medskip

\begin{enumerate}
	\item [(i)]
	$\upphi$-irreducible if there exists a $\sigma$-finite measure $\upphi$ on
	$\mathfrak{B}(\R^d)$ (the Borel $\sigma$-algebra on $\R^d$) such that whenever $\upphi(B)>0$ we have
	$\int_0^{\infty}\p(t,(x,i),B\times\{j\})\,\D t>0$ for all  $B\in\mathfrak{B}(\R^d)$, $(x,i)\in\R^d\times\mathbb{S}$ and $j\in\mathbb{S}$
	
	\medskip
	
	\item [(ii)]
	transient if it is $\upphi$-irreducible, and if there exists a countable
	covering of $\R^d$ with sets
	$\{B_k\}_{k\in\N}\subset\mathfrak{B}(\R^d)$, and for each
	$k\in\N$ there exists a finite constant $c_k\ge0$ such that
	$\int_0^{\infty}\p(t,(x,i),B_k\times\{j\})\,\D{t}\le c_k$ holds for all $(x,i)\in\R^d\times\mathbb{S}$ and $j\in\mathbb{S}$
	
	\medskip
	
	\item [(iii)]
	recurrent if it is $\upphi$-irreducible, and $\upphi(B)>0$ implies
	$\int_{0}^{\infty}\p(t,(x,i),B\times\{j\})\,\D{t}=\infty$ for all  $B\in\mathfrak{B}(\R^d)$,  $(x,i)\in\R^d\times\mathbb{S}$ and $j\in\mathbb{S}$.
\end{enumerate}

\medskip

\noindent Let us remark that if $\{(\X(x,i;t),\Lambda(x,i;t))\}_{t\ge0}$ is  $\upphi$-irreducible, then the irreducibility  measure $\upphi$ can be
maximized, that is, there exists a unique ``maximal'' irreducibility
measure $\uppsi$ such that for any measure $\Bar{\upphi}$,
$\{(\X(x,i;t),\Lambda(x,i;t))\}_{t\ge0}$ is $\Bar{\upphi}$-irreducible if, and only if,
$\Bar{\upphi}$ is absolutely continuous with respect to $\uppsi$ (see \cite[Theorem~2.1]{Tweedie-1994}).
In view to this, when we refer to an irreducibility
measure we actually refer to the maximal irreducibility measure.
It is also well known that every $\uppsi$-irreducible Markov
process is either transient or recurrent (see \cite[Theorem
2.3]{Tweedie-1994}). 
Further, $\{(\X(x,i;t),\Lambda(x,i;t))\}_{t\ge0}$ is said to be

\medskip

\begin{enumerate}
	\item[(i)] open-set irreducible
	if the support of its maximal irreducibility measure $\uppsi$,  
	$${\rm supp}\,\uppsi\,=\,\{x\in\R^d: \uppsi(O)>0\ \text{for every open neighborhood}\ O\ \text{of}\ x\}\,,$$ has a non-empty interior
	
	\medskip
	
	\item [(ii)] aperiodic if it admits an irreducible skeleton chain, that is, 
	there exist $t_0>0$ and a $\sigma$-finite measure $\upphi$ on
	$\mathfrak{B}(\R^d)$, such that $\upphi(B)>0$ implies
	$\sum_{n=0}^{\infty} \p(nt_0,(x,i),B\times\{j\}) >0$ for all $B\in\mathfrak{B}(\R^d)$, $(x,i)\in\R^d\times\mathbb{S}$ and $j\in\mathbb{S}$.
\end{enumerate}

\medskip

\noindent 
For $t\ge0$ and a  measure $\upmu$ on $\mathfrak{B}(\R^d)\times\mathscr{P}(\mathbb{S})$, $\upmu\mathcal{P}_t$ stands for $\int_{\R^{d}\times\mathbb{S}}\p(t,(x,i),\D y\times\{j\})\,\upmu(\D x\times\{i\}).$ Observe that $\updelta_{(x,i)}\mathcal{P}_t=\p(t,(x,i),\D y\times\{j\})$, where $\updelta_{(x,i)}$ denotes the Dirac measure at $(x,i).$
A
(not necessarily finite) measure $\uppi$ on $\mathfrak{B}(\R^d)\times\mathscr{P}(\mathbb{S})$ is called invariant for
$\{(\X(x,i;t),\Lambda(x,i;t))\}_{t\ge0}$ if
$\upmu\mathcal{P}_t=\upmu$
for all $t\ge0$. 
It is well known that if $\{(\X(x,i;t),\Lambda(x,i;t))\}_{t\ge0}$ is
recurrent, then it possesses a unique (up to constant
multiples) invariant measure $\uppi$
(see \cite[Theorem~2.6]{Tweedie-1994}).
If the 
invariant measure is
finite, then it may be normalized to a probability measure. If
$\{(\X(x,i;t),$\linebreak $\Lambda(x,i;t))\}_{t\ge0}$ is recurrent with finite invariant measure, then it is called 
positive recurrent; otherwise it is called null recurrent. It is easy to see that a transient 
 process cannot have a finite invariant measure. 
Finally, $\{(\X(x,i;t),\Lambda(x,i;t))\}_{t\ge0}$ is said to be ergodic
if it possesses an invariant probability 
measure $\uppi$ and there exists a nondecreasing function
$r:[0,\infty)\to[1,\infty)$ such that 
\begin{equation*}
\lim_{t\to\infty}r(t)\lVert \updelta_{(x,i)}\mathcal{P}_t
-\uppi\rVert_{{\rm TV}} \,=\,0
\end{equation*} for all $(x,i)\in\R^d\times\mathbb{S}$, where $\lVert\upmu\rVert_{{\rm TV}}:=\sup_{B\in\mathfrak{B}(\R^d)\times\mathscr{P}(\mathbb{S})}|\upmu(B)|$ is the total variation norm  of a signed measure $\upmu$ (on $\mathfrak{B}(\R^d)\times\mathscr{P}(\mathbb{S})$).
We say that $\{(\X(x,i;t),\Lambda(x,i;t))\}_{t\ge0}$ is sub-geometrically ergodic if it is ergodic and 
$\lim_{t\to\infty}\ln r(t)/t=0$, and 
that it is geometrically ergodic if it is ergodic and
$r(t)=\E^{\kappa t}$ for some $\kappa>0$. 
Let us remark that (under the assumptions of $\mathcal{C}_b$-Feller property, open-set irreducibility
and aperiodicity) ergodicity is equivalent 
to positive recurrence (see 
\cite[Theorem 6.1]{Meyn-Tweedie-AdvAP-II-1993},
and \cite[Theorems 4.1, 4.2 and 7.1]{Tweedie-1994}).

We now recall the notion and some general facts about Wasserstein distances  (on $\R^d\times\mathbb{S}$). Let $\uprho$ be a distance on $\R^d\times\mathbb{S}$. Denote by $\mathfrak{B}_\uprho(\R^d\times\mathbb{S})$  the  Borel $\sigma$-algebra on $\R^d\times\mathbb{S}$ induced by $\uprho$. 
For $p\ge0$ let $\mathcal{P}_{\uprho,p}$ be the space of all probability measures $\upmu$ on $\mathfrak{B}_\uprho(\R^d\times\mathbb{S})$ having finite $p$-th moment, that is, $\int_{\R^d\times\mathbb{S}}\uprho((x,i),(y,j))^p\upmu(\D y\times\{j\})<\infty$ for some (and then  any) $(x,i)\in\R^d\times\mathbb{S}$.
  For $p\ge1$ and  $\upmu,\upnu\in\mathcal{P}_{\uprho,p}$, the $\mathcal{L}^p$-Wasserstein distance between $\upmu$ and $\upnu$ is defined as $$\mathcal{W}_{\uprho,p}(\upmu,\upnu)\,\df\,\inf_{\Pi\in\mathcal{C}(\upmu,\upnu)}\left(\int_{(\R^{d}\times\mathbb{S})\times(\R^{d}\times\mathbb{S})}\uprho\bigl((x,i),(y,j)\bigr)^p\,\Pi\bigl(\D x\times\{i\},\D y\times\{j\}\bigr)\right)^{1/p}\,,$$ where  $\mathcal{C}(\upmu,\upnu)$ is the family of couplings of $\upmu$ and $\upnu$, that is, $\Pi\in\mathcal{C}(\upmu,\upnu)$ if, and only if, $\Pi$ is a probability measure on $(\R^{d}\times\mathbb{S})\times(\R^{d}\times\mathbb{S})$ having $\upmu$ and $\upnu$ as its marginals. It is not hard to see that $\mathcal{W}_{\uprho,p}$ satisfies the axioms of a (not necessarily finite) distance on $\mathcal{P}_{\uprho,p}$. The restriction  of $\mathcal{W}_{\uprho,p}$ to  $\mathcal{P}_{\uprho,p}$   defines a finite distance.
If $(\R^d\times\mathbb{S},\uprho)$ is a complete (separable) metric space space, then it is well known that $(\mathcal{P}_{\rho,p},\mathcal{W}_{\uprho,p})$ is also a complete (separable) metric space (see \cite[Theorem 6.18]{Villani-Book-2009}). Of our special interest will be the situation when $\uprho$ takes the form \begin{equation}\label{eq:rho}\uprho\bigl((x,i),(y,j)\bigr)\,=\,\mathbbm{1}_{\{i\neq j\}}(i,j)+f\bigl(|x-y|\bigr)\end{equation} for some
 non-decreasing concave $f:[0,\infty)\to[0,\infty)$  satisfying 
$f(u)=0$ if, and only if, $u=0$. In this situation, the corresponding Wasserstein space is denoted by $(\mathcal{P}_{f,p},\mathcal{W}_{f,p})$ (which is always a complete  metric space).
Observe that if $f(t)=\mathbb{1}_{(0,\infty)}(t)$, then $\mathcal{W}_{f,p}(\upmu,\upnu)=\rVert\upmu-\upnu\lVert_{{\rm TV}}$ for all $p\ge1$. 
For more on Wasserstein distances we refer the readers to \cite{Villani-Book-2009}.


\subsection{Main results} We now state the main results of this article.   
We first discuss subgeometric ergodicity  of $\{(\X(x,i;t),\Lambda(x,i;t))\}_{t\ge0}$ with respect to the total variation distance. Recall, a right-continuous temporally-homogeneous Markov chain $\{\Lambda(i;t)\}_{t\ge0}$ on $\mathbb{S}$ given by state-independent generator $\mathcal{Q}$ is irreducible if for any $i,j \in \mathbb{S}$, $i\neq j$, there are $m\in\N$ and $k_0, \dots, k_m \in \mathbb{S}$ with $k_0=i$, $k_m=j$ and $k_l \neq k_{l+1}$ for $l=0,\dots,m-1$,   such that 
$\mathrm{q}_{k_{l}k_{l+1}}>0$ for all $l=0,\dots,m-1$. Due to finiteness of $\mathbb{S}$, it is well known  that  $\{\Lambda(i;t)\}_{t\ge0}$ is  then geometrically ergodic. Let 	$\{\bar\Lambda(i;t)\}_{t\ge0}$ be an independent copy of $\{\Lambda(i;t)\}_{t\ge0}$. Put $$\tau_{ij}\df \left\{ 
\begin{array}{ll} 
\inf\{t>0\colon \Lambda(i;t)\,=\,\bar\Lambda(j;t)\}\,, & i\neq j\,, \\ 
0\,, & i=j\,,
\end{array} \right. $$
and $\zeta \df \inf_{i,j \in \mathbb{S}} \Prob(\Lambda(i;1)=j)$. Observe that  $0<\zeta<1$ (recall that $\mathbb{S}$ is finite and $\{\Lambda(i;t)\}_{t\ge0}$ is irreducible).
Define $ \vartheta\df -\log (1-\zeta)$. It holds that \begin{equation}\label{eq23}\Prob(\tau_{ij}>t)\,\le\,\E^{-\vartheta \lfloor t \rfloor}\end{equation}
for all $i,j\in\mathbb{S}$ and $t\ge0$ (see \Cref{lemma:theta}).

 \begin{theorem}\label{tm:TV}
	Let $\{(\X(x,i;t),\Lambda(x,i;t))\}_{t\ge0}$
	be an open-set irreducible and aperiodic regime-switching diffusion process satisfying 
	\textbf{(A1)}-\textbf{(A5)}. Assume 
	
	\medskip
	
	\begin{itemize}
		\item [(i)] there are $\{c_i\}_{i\in\mathbb{S}}\subset\R$, twice continuously differentiable $\mathsf{V}:\R^d\to(1,\infty)$  and twice continuously differentiable nondecreasing concave $\theta:(1,\infty)\to(0,\infty)$, 
		such that
		\begin{align*}\lim_{u\to \infty}\theta'(u)&\,=\,0\,,\qquad \limsup_{|x|\to\infty}\frac{\mathcal{L}_i\mathsf{V}(x)}{\theta\circ\mathsf{V}(x)}\,<\,c_i\,,\\ \lim_{|x|\to\infty}\frac{\theta\circ\mathsf{V}(x)}{\mathsf{V}(x)}&\,=\,0\,,\qquad \lim_{|x|\to\infty}\sup_{i\in\mathbb{S}}\frac{\mathcal{L}_i\theta\circ\mathsf{V}(x)}{\theta\circ\mathsf{V}(x)}\,=\,0 \end{align*}

		\medskip
		
		\item[(ii)] $\mathcal{Q}(x)=\mathcal{Q}+\mathrm{o}(1)$ as $|x|\to\infty$\footnote{We use the standard $\mathsf{o}$ notation: for $h:\R^p\to\R^q$  we write $h(x) =\mathsf{o}(1)$ as $|x|\to\infty$ if, and only if, $\lim_{|x|\to\infty}h(x)$ is the zero function.}, where $\mathcal{Q}=(\mathrm{q}_{ij})_{i,j\in\mathbb{S}}$ is the infinitesimal generator of an irreducible right-continuous temporally-homogeneous Markov chain  on $\mathbb{S}$ with invariant probability measure  $\uplambda=(\uplambda_i)_{i\in\mathbb{S}}$

\medskip

			\item[(iii)] 
		 $\displaystyle\sum_{i\in\mathbb{S}}c_i\uplambda_i<0$.
	\end{itemize}

\medskip

	\noindent Then, $\{(\X(x,i;t),\Lambda(x,i;t))\}_{t\ge0}$
	admits a unique invariant probability measure $\uppi$ and
	\begin{equation*}
	\lim_{t\to\infty}r(t)\lVert \updelta_{(x,i)}\mathcal{P}_t
	-\uppi\rVert_{{\rm TV}} \,=\,0
	\end{equation*} for all $(x,i)\in\R^d\times\mathbb{S}$,
	where $r(t)=\theta\circ \Theta^{-1}(t)$ with $$\Theta(t)\,=\,\int_1^t\frac{\D u}{\theta(u)}\,.$$ 
\end{theorem}

The   proof of Theorem \ref{tm:TV} is based on the Foster-Lyapunov method for subgeometric ergodicity of Markov processes developed in \cite{Douc-Fort-Guilin-2009}.
The method itself consists of finding an appropriate recurrent (petite) set  $C\in\mathcal{B}(\R^d)\times\mathcal{P}(\mathbb{S})$ and constructing an appropriate  function $\mathcal{V}:\R^d\times\mathbb{S}\to[1,\infty)$ (the so-called Lyapunov (energy) function), such that the Lyapunov equation $$
\mathcal{L}\mathcal{V}(x,i)\,\le\,-\theta\circ\mathcal{V}(x,i)+\kappa\,\mathbb{1}_C(x,i)$$ holds for some $\kappa\in\R$ (see \cite[Theorems 3.2 and 3.4]{Douc-Fort-Guilin-2009}). Under the assumptions of the theorem (in particular, open-set irreducibility and aperiodicity of the process), we show that $C$ is of the form $K\times\mathbb{S}$ for some compact set $K\subset\R^d$, and $\mathcal{V}(x,i)$ is given in terms of $\{c_i\}_{i\in\mathbb{S}}$, $\theta(u)$ and $\mathsf{V}(x)$. 

Standard and crucial assumption ensuring open-set irreducibility and aperiodicity of\linebreak $\{(\X(x,i;t),\Lambda(x,i;t))\}_{t\ge0}$ used in the literature  is uniform ellipticity of the matrix $\upsigma(x,i)\upsigma(x,i)^T$ (see \cite{Kunwai-Zhu-2020} and the references therein). In \Cref{tm:irred} we relax this assumption and show that $\{(\X(x,i;t),\Lambda(x,i;t))\}_{t\ge0}$ will be open-set irreducible and aperiodic if $\upsigma(x,i)\upsigma(x,i)^T$ is uniformly elliptic on an open ball only, while on the rest of the state space it can degenerate. 
In the case when $\upsigma(x,i)\upsigma(x,i)^T$ is highly degenerated (for example, it  completely vanishes), the topology induced by the total variation distance becomes too ``rough'', that is, it cannot completely capture the singular behavior of $\{(\X(x,i;t),\Lambda(x,i;t))\}_{t\ge0}$,  and $\p(t,(x,i),\D y\times\{j\})$ cannot converge to the underlying invariant probability measure (if it exists) in this topology, but in a weaker sense. Therefore, in this situation, we naturally resort to Wasserstein distances which, in a certain sense, induce a finer topology (see \cite{Sandric-RIM-2017} and \cite{Villani-Book-2009}).
In the following result we first discuss asymptotic flatness (uniform dissipativity) of the semigroup of $\{(\X(x,i;t),\Lambda(x,i;t))\}_{t\ge0}$.

\begin{theorem}\label{tm:WASS-subgeom}
	Assume $\textbf{(A1)-(A5)}$, and suppose $\{\Lambda(x,i;t)\}_{t\ge0}$ and $\upsigma(x,i)$ are $x$-independent. 
	Assume also that $\{\Lambda(i;t)\}_{t\ge0}$ is irreducible and let $\uplambda=(\uplambda_i)_{i\in\mathbb{S}}$ be its invariant probability measure. Further, let $f,\psi:[0, +\infty) \to [0, +\infty)$ be such that
	
	\medskip
	
	\begin{enumerate}[(i)]
		\item $f(u)$ is bounded, concave, non-decreasing, absolutely continuous on $[u_0,u_1]$, for all $0<u_0<u_1<+\infty$, and $f(u)=0$ if, and only if, $u=0$
			
			\medskip

		\item $\psi(u)$ is convex and $\psi(u)=0$ if, and only if, $u=0$
		
			\medskip
		
		\item there are $\{\Gamma_i\}_{i\in\mathbb{S}}\subset(-\infty,0]$  such that
		\begin{equation}\label{eq:WASS-fin-bdd1}
		f'\bigl(|x-y|\bigr) \bigl\langle x-y, b(x,i)-b(y,i) \bigr\rangle \,\leq\, 
		\Gamma_i |x-y| \,\psi\bigl(f(|x-y|)\bigr)
		\end{equation}
		a.e. on $\R^d$
		
			\medskip

		\item $\displaystyle \sum_{i\in\mathbb{S}} \Gamma_i \uplambda_i< 0$.
	\end{enumerate}

\medskip

\noindent	Then, 
for $\uprho$ given by \cref{eq:rho}, and all $p\ge1$ and  $(x,i),(y,j) \in \R^d\times \mathbb{S}$ it holds that
\begin{equation}
\label{eq:subb} \lim_{t\to\infty}\mathcal{W}_{f,p}\bigl(\updelta_{(x,i)}\mathcal{P}_t, \updelta_{(y,j)}\mathcal{P}_t\bigr)\, =\, 0\,.\end{equation}
		Additionally, if $\psi(u)=u^q$ for some $q>1$, then
	$$\lim_{t \to \infty}t^{1/(q-1)}\mathcal{W}_{f,p}\bigl(\updelta_{(x,i)}\mathcal{P}_t, \updelta_{(y,j)}\mathcal{P}_t\bigr)\, \leq\,\left(  \frac{1-q}{2}\sum_{i \in \mathbb{S}} \Gamma_i \uplambda_i\right)^{1/(1-q)}\,.$$
If $\psi(u)=\kappa u$ for some $\kappa>0$, then
$$\lim_{t \to \infty}\E^{\alpha t/2}\,\mathcal{W}_{f,p}\bigl(\updelta_{(x,i)}\mathcal{P}_t, \updelta_{(y,j)}\mathcal{P}_t\bigr)\, =\,0$$ for all $0<\alpha<\min\{\vartheta/p,-\kappa\sum_{i \in \mathbb{S}} \Gamma_i \uplambda_i\}$, where $\vartheta$ is given in \cref{eq23}.	
	
	\end{theorem}

Crucial assumption in \Cref{tm:WASS-subgeom} is that the function $f(u)$, that is, distance $\uprho$, is bounded. In the following theorem we discuss the situation when this is not necessarily the case. 

\begin{theorem}\label{tm:WASS-subgeom1.1}
	Assume $\textbf{(A1)-(A5)}$, and suppose $\{\Lambda(x,i;t)\}_{t\ge0}$ and $\upsigma(x,i)$ are $x$-independent. 
	Assume also that $\bb(x,i)$ is locally Lipschitz continuous for every $i\in\mathbb{S}$, and that $\{\Lambda(i;t)\}_{t\ge0}$ is irreducible and let $\uplambda=(\uplambda_i)_{i\in\mathbb{S}}$ be its invariant probability measure. Further,  assume that  there is  $0<K<\vartheta$ (recall that $\vartheta$ is given in \cref{eq23}) such that
	\begin{equation}\label{linear} 2\bigl\langle x, \mathrm{b} (x,i) \bigr\rangle + \mathrm{Tr} \bigl(\upsigma(i) \upsigma(i)^T\bigr)\,\leq\, K (1+ |x|^2)\end{equation} for all $x \in \R^d$ and $i \in \mathbb{S}.$
	Let $f,\psi:[0, +\infty) \to [0, +\infty)$ be such that
	\begin{enumerate}[(i)]
		\item $f(u)$ is concave, non-decreasing, absolutely continuous on $[u_0,u_1]$, for all $0<u_0<u_1<\infty$, and $f(u)=0$ if, and only if, $u=0$

		\medskip
		
		\item $\psi(u)$ is convex and $\psi(u)=0$ if, and only if, $u=0$
		
		\medskip
		
		\item  there are $\{\Gamma_i\}_{i\in\mathbb{S}}\subset(-\infty,0]$ and $\eta > \inf \{ f(u) \mid u > 0 \}$, such that
		\begin{equation}\label{eq:WASS-fin-unbdd1}
		f'(|x-y|) \bigl\langle x-y, b(x,i)-b(y,i) \bigr\rangle\, \leq\, \left\{ 
		\begin{array}{ll} 
		\Gamma_i |x-y| \psi\bigl(f(|x-y|)\bigr)\,, &   f(|x-y|) \leq \eta\,, \\ 
		0\,, & \text{otherwise}
		\end{array} \right.
		\end{equation}
		a.e. on $\R^d$
		
			\medskip

		\item $\displaystyle \sum_{i\in\mathbb{S}} \Gamma_i \uplambda_i< 0$.

	\end{enumerate}
	
\noindent 	Then, for $\uprho$ given by \cref{eq:rho}, and all   $(x,i),(y,j) \in \R^d\times \mathbb{S}$ it holds that
	\begin{equation}
	\label{eq:subb1} \lim_{t\to\infty}\mathcal{W}_{f,1}\bigl(\updelta_{(x,i)}\mathcal{P}_t, \updelta_{(y,j)}\mathcal{P}_t\bigr)\, =\, 0\,.\end{equation}
Additionally, if $\psi(u)=u^q$ for some $q>1$, then $\mathbb{E} \left[ |\X(x,i;\tau_{ij})-\X(y,j;\tau_{ij})|^2\right]<\infty$ and
\begin{align*}&\lim_{t \to \infty}t^{1/(q-1)}\mathcal{W}_{f,1}\bigl(\updelta_{(x,i)}\mathcal{P}_t, \updelta_{(y,j)}\mathcal{P}_t\bigr)\\&\, \leq\,\mathbb{E} \left[\lceil \delta |\X(x,i;\tau_{ij})-\X(y,j;\tau_{ij})| \rceil^2\right]^{1/2}\left(  \frac{1-q}{2}\sum_{i \in \mathbb{S}} \Gamma_i \uplambda_i\right)^{1/(1-q)}\,,\end{align*}
where $\delta \df \inf \{ t\geq0 \colon f(1/t) \leq \eta \}.$
If $\psi(u)=\kappa u$ for some $\kappa>0$, then
$$\lim_{t \to \infty}\E^{\alpha t/2}\,\mathcal{W}_{f,1}\bigl(\updelta_{(x,i)}\mathcal{P}_t, \updelta_{(y,j)}\mathcal{P}_t\bigr)\, =\,0$$ for all $0<\alpha<\min\{\vartheta,-\kappa\sum_{i \in \mathbb{S}} \Gamma_i \uplambda_i\}$.
	\end{theorem}

As a consequence of \Cref{tm:WASS-subgeom,tm:WASS-subgeom1.1} we conclude the following ergodic behavior of\linebreak $\{(\X(x,i;t),\Lambda(x,i;t))\}_{t\ge0}$.

\begin{theorem}\label{tm:WASS-subgeom2}
	In addition to  the assumptions of \Cref{tm:WASS-subgeom} or \Cref{tm:WASS-subgeom1.1},
	suppose that  there are non-negative $\mathcal{V}\in\mathcal{C}^2(\R^d\times\mathbb{S})$ and locally bounded  $g:\R^d\to\R$, such that \begin{equation}\label{eq:eq}\lim_{|x|\to\infty}g(x)\,=\,\infty\qquad\text{and}\qquad \mathcal{L}\mathcal{V}(x,i)\,\le\, -g(x)\end{equation} for all $(x,i)\in\R^d\times\mathbb{S}$.
		Then, 	$\{(\X(x,i;t),\Lambda(x,i;t))\}_{t\ge0}$ admits a unique invariant probability measure $\uppi$ and
	$$ \lim_{t\to\infty}\mathcal{W}_{f,p}\bigl(\updelta_{(x,i)}\mathcal{P}_t, \uppi\bigr)\, =\, 0$$  for all $(x,i)\in\R^d\times\mathbb{S}$.
	Additionally, if $\psi(u)=u^q$ for some $q>1$, then 
	$$\lim_{t \to \infty}t^{1/(q-1)}\mathcal{W}_{f,p}\bigl(\updelta_{(x,i)}\mathcal{P}_t, \uppi\bigl)\, \leq\,\left(  \frac{1-q}{2}\sum_{i \in \mathbb{S}} \Gamma_i \uplambda_i\right)^{1/(1-q)}\,,$$ and
	if $\psi(u)=\kappa u$ for some $\kappa>0$, then 
	$$\lim_{t \to \infty}\E^{\alpha t/2}\,\mathcal{W}_{f,p}\bigl(\updelta_{(x,i)}\mathcal{P}_t, \uppi\bigl)\, =\,0$$ for all   $0<\alpha<\min\{\vartheta/p,-\kappa\sum_{i \in \mathbb{S}} \Gamma_i \uplambda_i\}$.	Recall that in the case of \Cref{tm:WASS-subgeom1.1}  $p=1$.
\end{theorem}

Finally, we discuss sufficient conditions ensuring \cref{eq:eq}. First, recall that an
$m\times m$ matrix $M$ is called an $\mathcal{M}$-matrix if it can be expressed as $M = \gamma\,\mathbb{I}_m-N$ for some $\gamma > 0$ and some nonnegative $m\times m$ matrix $N$ with the property that $\varrho(N) \le \gamma$, where $\mathbb{I}_m$ and $\varrho(N)$ denote the $m\times m$ identity matrix and spectral radius of $N$. 
According to the  Perron-Frobenius theorem, $M$ is nonsingular if, and only if, $\varrho(N) < \gamma$.

\begin{theorem}\label{tm:WASS-subgeom3}
Assume $\textbf{(A1)-(A5)}$ and
$\mathcal{Q}(x)=\mathcal{Q}+\mathrm{o}(1)$. 
Furthermore, assume that  there are $\{c_i\}_{i\in\mathbb{S}}\subset\R$  such that  either one of the following conditions holds:

\medskip

\begin{itemize}
	\item [(i)] $\mathcal{Q}$ is the infinitesimal generator of an irreducible right-continuous temporally-homogeneous Markov chain  on $\mathbb{S}$ with invariant probability measure  $\uplambda=(\uplambda_i)_{i\in\mathbb{S}}$,  $\sum_{i\in\mathbb{S}} c_i \uplambda_i< 0$ and there are  twice continuously differentiable $\mathsf{V}:\R^d\to(0,\infty)$  and twice continuously differentiable concave  $\theta:(0,\infty)\to(0,\infty)$, 
	such that
	\begin{align*}\lim_{|x|\to \infty}\theta\circ\mathsf{V}(x)&\,=\,\infty\,,\qquad \limsup_{|x|\to\infty}\frac{\mathcal{L}_i\mathsf{V}(x)}{\theta\circ\mathsf{V}(x)}\,<\,c_i\,,\\ \lim_{|x|\to\infty}\frac{\theta\circ\mathsf{V}(x)}{\mathsf{V}(x)}&\,=\,0\,,\qquad \lim_{|x|\to\infty}\sup_{i\in\mathbb{S}}\frac{\mathcal{L}_i\theta\circ\mathsf{V}(x)}{\theta\circ\mathsf{V}(x)}\,=\,0 \,.\end{align*}
	
	\medskip
	
	\item[(ii)] $\mathcal{Q}$ is the infinitesimal generator of an irreducible right-continuous temporally-homogeneous Markov chain  on $\mathbb{S}$ with invariant probability measure  $\uplambda=(\uplambda_i)_{i\in\mathbb{S}}$, $\sum_{i\in\mathbb{S}} c_i \uplambda_i< 0$ and there is a twice continuously differentiable $\mathsf{V}:\R^d\to(0,\infty)$ such that
	\begin{equation}\label{eq:n}\lim_{|x|\to \infty}\mathsf{V}(x)\,=\,\infty\qquad\text{and}\qquad \limsup_{|x|\to\infty}\frac{\mathcal{L}_i\mathsf{V}(x)}{\mathsf{V}(x)}\,<\,c_i\,.\end{equation} 
	
	\medskip
	
	\item[(iii)] $(-\mathcal{Q}+\mathrm{diag}\,c)$ is a non-singular $\mathcal{M}$-matrix, where $c=(c_i)_{i\in\mathbb{S}}$, and there is a twice continuously differentiable $\mathsf{V}:\R^d\to(0,\infty)$ satisfying \cref{eq:n}.\end{itemize}
\medskip	Then there  are non-negative $\mathcal{V}\in\mathcal{C}^2(\R^d\times\mathbb{S})$ and locally bounded  $g:\R^d\to\R$, such that \cref{eq:eq} holds.	
\end{theorem}


\subsection{Literature review}  
Our work relates to the active research on ergodicity properties of Markov processes, with focus  on regime-switching diffusion processes. This type of processes have a significant role in modeling of many phenomena arising in nature and  engineering (for instance, see \cite{Crudu-Debussche-Muller-Radulescu-2012}, \cite{Ghosh-Arapostathis-Marcus-1993}, \cite{Guo-Zhang-2004} and \cite{Ji-Chizeck-1990}).
At first glance one may think that these processes behave like the standard  diffusion processes. However, the underlying regime-switching component  can considerably influence  some of their  properties. In \cite{Pinsky-Scheutzow-1992} it is shown 
that even  if in every fixed regime $i\in\mathbb{S}$ the corresponding diffusion process is positive recurrent (transient), the overall regime-switching diffusion   can be transient (positive recurrent).
This shows that the behavior of these type of processes
depends on an (sometime very nontrivial) interplay between both (continuous-state and regime-switching) components.

By using the   classical Foster-Lyapunov method  for geometric ergodicity of Markov processes,
in \cite{Chen-Chen-Tran-Yin-2019},\cite{Cloez-Hairer-2015}, \cite{Kunwai-Zhu-2020}, \cite{Mao-Yuan-Book-2006}, \cite{Nguyen-Yin-2018},  \cite{Shao-Xi-2013}, \cite{Shao-2015} , \cite{Tong-Majda-2016},  \cite{Xi-2009}, \cite{Xi-Zhu-2017} and  \cite{Yin-Zhu-Book-2010} 
geometric ergodicity  with respect to the total variation distance of regime-switching diffusions  is established.  In this article, we employ the Foster-Lyapunov method for subgeometric ergodicity of Markov processes developed in \cite{Douc-Fort-Guilin-2009} and   obtain  conditions  ensuring subgeometric ergodicity of this class of processes.  Furthermore, we adapt these results and also discuss  subgeometric ergodicity  of a class of regime-switching Markov processes  obtained through the Bochner's subordination.

As  mentioned above, the studies on ergodicity properties with respect to the total variation distance assume that the underlying process is irreducible and aperiodic. This is satisfied if the process does not show a singular behavior in its motion, for example, if its diffusion part is non-singular. For Markov processes that do not converge in total variation, ergodic properties under Wasserstein distances are studied since this distance function, in a certain sense, induces a finner topology (see \cite{Sandric-RIM-2017} and \cite{Villani-Book-2009}).
 In \cite{Cloez-Hairer-2015}, \cite{Shao-2015-2} and \cite{Tong-Majda-2016} the coupling approach together with the Foster-Lyapunov method is employed to establish geometric contractivity and ergodicity of the semigroup of a regime-switching diffusion with respect to a Wasserstein distance. 
 In this article, by using the ideas developed in \cite{Lazic-Sandric-2021},  in the context of diffusion processes, and combining the asymptotic flatness conditions in \cref{eq:WASS-fin-bdd1,eq:WASS-fin-unbdd1} and the coupling method, we obtain subgeometric contractivity and  ergodicity of the semigroup of a regime-switching diffusion, with possibly singular diffusion coefficient, with respect to a Wasserstein distance. At the end, we again discuss subgeometric ergodicity, but with respect to Wasserstein distances, of a class of regime-switching diffusions with jumps and a class of regime-switching Markov processes obtained through the Bochner's subordination.
These results are related to \cite{Chen-Chen-Tran-Yin-2019} and \cite{Kunwai-Zhu-2020} where  geometric ergodicity  of regime-switching diffusions  with jumps is established with respect to the total variation distance.


\subsection{Organization of the article} 
In the next section, we discuss open-set irreducibility and aperiodicity of the process $\{(\X(x,i;t),\Lambda(x,i;t))\}_{t\ge0}$. In \Cref{S3}, we prove \Cref{tm:TV}, and in  \Cref{S4} we prove \Cref{tm:WASS-subgeom,tm:WASS-subgeom1.1,tm:WASS-subgeom2,tm:WASS-subgeom3}. In the last section, \Cref{S5}, we briefly discuss ergodicity properties of a class of regime-switching Markov processes with jumps.
 

\section{Irreducibility and aperiodicity}\label{S2}

In this section, we discuss irreducibility and aperiodicity properties of $\{(\X(x,i;t),\Lambda(x,i;t))\}_{t\ge0}$.
Fix $i \in \mathbb{S}$ and consider the following stochastic differential equation (SDE):
\begin{equation}
\begin{aligned}\label{eq2}
\D \X^{(i)}(x;t)&\,=\,\bb\bigl(\X^{(i)}(x;t),i\bigr)\D t+\upsigma\bigl(\X^{(i)}(x;t),i\bigr)\D \B(t)\\  \X^{(i)}(x;0)&\,=\,x\in\R^d\,.
\end{aligned}
\end{equation}
Assume the following:

\medskip

\begin{description}
	\item[($\overline{\textbf{A1}}$)]  for each     $i\in\mathbb{S}$ and $x\in\R^d$  the  SDE in  \cref{eq2} admits a unique nonexplosive strong solution $\{\X^{(i)}(x;t)\}_{t\ge0}$ which has continuous sample paths and it is a temporally-homogeneous strong Markov process with transition kernel $\p^{(i)}(t,x,\D y)=\Prob(\X^{(i)}(x;t))\in\D y)$
	
	\medskip
	
	\item[($\overline{\textbf{A2}}$)] there is $\Delta>0$ such that for each  $i\in\mathbb{S}$, 	$\inf_{x\in\R^d}\mathrm{q}_{ii}(x)>-\Delta$
	
	\medskip
	

\item[($\overline{\textbf{A3}}$)] there are $\Delta,\alpha>0$ such that for all $x,y\in\R^d$ and $i\in\mathbb{S},$ 
$$\sum_{j\in\mathbb{S}\setminus\{i\}}|\mathrm{q}_{ij}(x)-\mathrm{q}_{ij}(y)|\,\le\, \Delta\,|x-y|^\alpha$$

\medskip

	\item[($\overline{\textbf{A4}}$)]   there are $x_0 \in \R^d$ and $r_0>0$, such that 
	
	\smallskip
	
	\begin{enumerate}[(i)]
		\item there are $\alpha, \Delta, \delta >0$, such that for all $x,y \in \mathscr{B}_{r_0}(x_0)$ and $i \in \mathbb{S}$,
		$$|\bb(x,i)-\bb(y,i)|+\lVert  \upsigma(x,i) - \upsigma(y,i) \rVert_{\rm HS}\,  \leq\, \Delta \,|x-y|^\alpha \quad \text{and} \quad |\upsigma(x,i)'y|\, \geq \, \delta\,|y|$$
		
			\smallskip
		
		\item [(ii)] for all  $i\in\mathbb{S}$ and $x\in\R^d$,
		$$\mathbb{P}\bigl(\uptau^{(x,i)}_{\mathscr{B}_{r_0}(x_0)}<\infty\bigr)\,>\,0\,,$$  where $\uptau^{(x,i)}_{\mathscr{B}_{r_0}(x_0)}\df\inf \{t\geq0 : \X^{(i)}(x;t) \in \mathscr{B}_{r_0}(x_0) \}$.
	\end{enumerate}
\end{description}
We remark that conditions from     \cite{Xi-Yin-Zhu-2019} ensuring \textbf{(A1)}-\textbf{(A5)}  imply also \textbf{($\overline{\textbf{A1}}$)}. We now prove the main result of this section.

\begin{theorem}\label{tm:irred}
	Assume \textbf{(A1)}-\textbf{(A5)}, \textbf{($\overline{\textbf{A1}}$)}-\textbf{($\overline{\textbf{A4}}$)} and that
	   for any $i,j \in \mathbb{S}$, $i\neq j$, there are $n\in\N$ and $k_0, \dots, k_n \in \mathbb{S}$ with $k_0=i$, $k_n=j$ and $k_l \neq k_{l+1}$ for $l=0,\dots,n-1$,   such that the set
	 $\{x\in\mathscr{B}_{r_0}(x_0)\colon
	 \mathrm{q}_{k_{l}k_{l+1}}(x)>0\}$ has positive Lebesgue measure for all $l=0,\dots,n-1$.
	Then the process $\{(\X(x,i;t),\Lambda(x,i;t))\}_{t\ge0}$ is open-set irreducible and aperiodic.
\end{theorem}

\begin{proof}
	We show that for all $(x,i)\in\R^d\times\mathbb{S}$ and $j\in\mathbb{S}$, $$\int_0^\infty \p\bigl(t,(x,i),B\times\{j\}\bigr)\,\D t\,>\,0 \qquad \text{and}\qquad  \sum_{m=0}^\infty \p\bigl(m,(x,i),B\times\{j\}\bigr)\,>\,0$$ whenever $\mathrm{Leb}(B\cap\mathscr{B}_{r_0}(x_0))>0$. In other words, we show that $\{(\X(x,i;t),\Lambda(x,i;t))\}_{t\ge0}$ and the skeleton chain  $\{(\X(x,i;m),\Lambda(x,i;m))\}_{m\ge0}$ are $\upphi$-irreducible with $\upphi(\cdot)\df\mathrm{Leb}(\cdot\cap \mathscr{B}_{r_0}(x_0)).$
	
From
 the proof of \cite[Theorem 4.8]{Xi-Yin-Zhu-2019} we have that
 \begin{equation}
\begin{aligned}\label{eq:irr}
&\p\bigl(t,(x,i), B \times \{ j \} \bigr) \,\ge\, 
\updelta_{ij}\, \Prob\bigl( \X^{(i)}(x;t)\in B\bigr)\, \E^{-\Delta t}\\ &+ (1-\updelta_{ij})\,\E^{-\Delta t}\sum_{m=1}^\infty \underset{0<t_1<\ldots<t_m<t}{\idotsint} \underset{\substack{k_0,\ldots,k_m \in \mathbb{S} \\ k_l \neq k_{l+1}\\ k_0=i,\,  k_m=j}}{\sum} \int_{\R^d} \cdots \int_{\R^d} \Prob\bigl( \X^{(k_0)}(x;t_1)\in \D y_1\bigr)\, \mathrm{q}_{k_0k_1}(y_1) &\\
& \Prob\bigl( \X^{(k_1)}(y_1;t_2-t_1)\in \D y_2\bigr) \mathrm{q}_{k_1k_2}(y_2) \cdots \mathrm{q}_{k_{m-1}k_m}(y_m) \Prob\bigl( \X^{(k_m)}(y_m;t-t_m)\in B\bigr)\, \D t_1 \ldots \D t_m\,,
\end{aligned}
\end{equation}
where $\updelta_{ij}$ is the Kronecker delta.
In particular,
\begin{align*} 
&\p\bigl(t,(x,i), B \times \{ j \} \bigr) \,\ge\, 
\updelta_{ij}\, \Prob\bigl( \X^{(i)}(x;t)\in B\bigr)\,\E^{-\Delta t} \\ &+ (1-\updelta_{ij})\,\E^{-\Delta t} \underset{0<t_1<\ldots<t_m<t}{\idotsint}  \int_{\mathscr{B}_{r_0}(x_0)} \cdots \int_{\mathscr{B}_{r_0}(x_0)} \Prob\bigl( \X^{(k_0)}(x;t_1)\in \D y_1\bigr)\, \mathrm{q}_{k_0k_1}(y_1) &\\
& \hspace{1.6cm}\Prob\bigl( \X^{(k_1)}(y_1;t_2-t_1)\in \D y_2\bigr)\, \mathrm{q}_{k_1k_2}(y_2) \cdots \mathrm{q}_{k_{m-1}k_m}(y_m) \Prob\bigl(\X^{(k_m)}(y_m;t-t_m)\in B\bigr)\, \D t_1 \ldots \D t_m\,,
\end{align*}
where for $i\neq j$, $m\in\N$ and $k_0,\dots,k_m\in\mathbb{S}$ are such that $k_0=i$, $k_m=j$, and $k_l\neq k_{l+1}$ and $$\mathrm{Leb}\bigl(\{x\in\mathscr{B}_{r_0}(x_0)\colon
\mathrm{q}_{k_{l}k_{l+1}}(x)>0\}\bigr)\,>\,0$$ for $l=0,\dots,m-1$.
Next, assumptions of the theorem together with
the proof of \cite[Theorem 2.3]{Lazic-Sandric-2021} imply that for $B\in\mathfrak{B}(\R^d)$ with $\mathrm{Leb}(B\cap\mathscr{B}_{r_0}(x_0))>0$,

\medskip

\begin{itemize}
	\item [(i)] $ \displaystyle \Prob\bigl( \X^{(i)}(x;t)\in B\bigr)>0$ for all $x\in\mathscr{B}_{r_0}(x_0)$, $i\in\mathbb{S}$ and $t>0$
	
	\medskip
	
	\item [(ii)] $\displaystyle\int_0^\infty \Prob\bigl( \X^{(i)}(x;t)\in B\bigr)\,\D t>0$ for all $i\in\mathbb{S}$ and $x\in\R^d$
	
	\medskip
	
	\item[(iii)] $\displaystyle\sum_{m=0}^\infty \Prob\bigl( \X^{(i)}(x;m)\in B\bigr)>0$ for all $i\in\mathbb{S}$ and $x\in\R^d$.
\end{itemize}

\medskip

\noindent Thus, for all $(x,i)\in\R^d\times\mathbb{S}$, $j\in\mathbb{S}$  and $B\in\mathfrak{B}(\R^d)$ with $\mathrm{Leb}(B\cap\mathscr{B}_{r_0}(x_0))>0$,
$$\int_0^\infty \p\bigl(t,(x,i), B \times \{ j \} \bigr) \,\D t\,>\,0\qquad \text{and}\qquad\sum_{m=0}^\infty \p\bigl(m,(x,i), B \times \{ j \} \bigr) \,>\,0\,,$$ which proves the assertion.
\end{proof}

Let us now give several remarks.

\begin{remark} \label{Rem}
	\begin{itemize}
		\item [(i)] The conclusion of \Cref{tm:irred} also remains true in the case of  infinitely countable state space $\mathbb{S}$ (say $\mathbb{S}=\N$) if, in addition to the assumptions of the theorem, there is $\gamma>0$ such that for all  $i,j\in\mathbb{S}$, $i\neq j$, $\sup_{x\in\R^d}\mathrm{q}_{ij}(x)\le \gamma j3^{-j}.$ This additional assumption is required to conclude the relation in \cref{eq:irr} (see \cite[Lemma 4.7]{Xi-Yin-Zhu-2019}).
		
		\medskip
		
		\item[(ii)] The problem of open-set irreducibility and aperiodicity of regime-switching diffusion processes has already been considered in the literature (see \cite{Kunwai-Zhu-2020} and the references therein). In all these works a crucial assumption is uniform ellipticity of the matrix $\upsigma(x,i)\upsigma(x,i)^T$, that is, \begin{equation}\label{eq:UE}\inf_{x\in\R^d,\,y\in\R^d\setminus\{0\},\, i\in\mathbb{S}}\frac{|\upsigma(x,i)^Ty|}{|y|}\,>\,0\,.\end{equation}
		On the other hand, in \Cref{tm:irred} we  require  uniform ellipticity of  $\upsigma(x,i)\upsigma(x,i)^T$ on the open ball $\mathscr{B}_{r_0}(x_0)$ only, while on the rest of the state space it can degenerate provided $\mathbb{P}(\uptau^{(x,i)}_{\mathscr{B}_{r_0}(x_0)}<\infty)>0$ for all $i\in\mathbb{S}$ and $x\in\R^d$. 
Let us also remark  that
 irreducibility and aperiodicity (and ergodicity with respect to the total variation distance) of regime-switching diffusion processes can be  discussed by employing
 geometric control theory and differential geometry methods,  see for instance \cite{Dieu-Nguyen-Du-Yin-2016}.
		
		\medskip
		
		\item[(iii)] According to \cite[Proposition 2.4]{Lazic-Sandric-2021},  \textbf{($\overline{\textbf{A4}}$)} (ii) will hold if $|\upsigma(x,i)^Ty|>0$ for all $i\in\mathbb{S}$, $x\in\R^d$ and $y\in\R^d\setminus\{0\}$. Clearly, this condition is much weaker than \eqref{eq:UE}. A simple  example of a regime-switching diffusion process satisfying \textbf{($\overline{\textbf{A4}}$)} (as well as \textbf{(A1)}-\textbf{(A5)} and \textbf{($\overline{\textbf{A1}}$)}-\textbf{($\overline{\textbf{A3}}$)})
		with degenerate $\upsigma(x,i)\upsigma(x,i)^T$ is given as follows. Let $\mathbb{S}_>\subseteq\mathbb{S}$, and let $b_1\in\mathcal{C}^1(\R^d\times\mathbb{S}_>)$, $b_2(x,i)=(b_2^{(1)}(x,i),\dots,b_2^{(d)}(x,i))$ with $b_2^{(k)}\in\mathcal{C}^1(\R^d\times\mathbb{S}^c_>)$, $k=1,\dots,d$, and $\sigma\in\mathcal{C}^1(\R^d\times\mathbb{S})$ be such that:
		
		\medskip
		 
		\begin{itemize}
			\item [(a)] $0<\displaystyle\inf_{x\in\R^d} b_1(x,i)\le\displaystyle\sup_{x\in\R^d} b_1(x,i)<\infty$ for all $i\in\mathbb{S}_>$;
			
		\medskip
			
			\item[(b)] $\displaystyle\sup_{x\in\R^d\setminus\{0\}} \langle x,b_2(x,i)\rangle/|x|^2<\infty$ for all $i\in\mathbb{S}^c_>$;
			
	\medskip
			
			\item[(c)] for all   $i\in\mathbb{S}_>$, $\sigma(x,i)>0$ if, and only if, $x\in\mathscr{B}_{r_i}(0)$ for some  $r_i>0$;
			
						\medskip
			
			\item[(d)] $\sigma(x,i)>0$ for all $x\in\R^d$ and  $i\in\mathbb{S}_>^c.$
			
		\end{itemize}
		
\medskip
		
		\noindent Define $$\bb(x,i) \,\df\, \left\{
		\begin{array}{ll}
		-b_1(x,i)x\,, & i\in\mathbb{S}_>\,, \\
		b_2(x,i)\,, & i\in\mathbb{S}_>^c\,,
		\end{array} 
		\right.$$ and  $\upsigma(x,i)\df \sigma(x,i)\,\Id_{d}$, where $\Id_d$ stands for the $d\times d$ identity matrix. It is clear  that such a regime-switching diffusion process satisfies \textbf{($\overline{\textbf{A4}}$)}  with $x_0=0$ and any $0<r_0<\min_{i\in\mathbb{S}_>} r_i.$
	\end{itemize}
\end{remark}


 \section{Ergodicity with respect to the total variation distance}\label{S3}
 In this section, we prove \Cref{tm:TV}.

\begin{proof}[Proof of \Cref{tm:TV}] 
	Let $\beta\df -\sum_{i\in\mathbb{S}}c_i\uplambda_i>0$. Clearly, $$\sum_{i\in\mathbb{S}}(c_i+\beta)\, \uplambda_i\,=\,0\,.$$  From \cite[Lemma A.12]{Yin-Zhu-Book-2010} it then follows that the system $$\sum_{j\in\mathbb{S}}\mathrm{q}_{ij}\gamma_j\,=\,-c_i-\beta\,,\qquad i\in\mathbb{S}\,,$$ admits a solution $(\gamma_i)_{i\in\mathbb{S}}$. 
	Let $m>\max\{\beta,2\}$ and let $\mathcal{V}:\R^d\times\mathbb{S}\to[1,\infty)$ be twice continuously differentiable (in the first coordinate) and such that
	 $$\mathcal{V}(x,i)\,=\,\frac{m}{\beta}\bigl(\mathsf{V}(x)+\gamma_i\,\theta\circ\mathsf{V}(x)\bigr)$$  for all $i\in\mathbb{S}$ and all $|x|$ large enough.  Observe that the assumptions in (i) ensure existence of such a function.
	 For all $i\in\mathbb{S}$ and all $|x|$ large enough, we now have
	\begin{align*}\mathcal{L}\mathcal{V}(x,i)&\,=\,\frac{m}{\beta}\mathcal{L}_i\mathsf{V}(x)+\frac{m\gamma_i}{\beta}\mathcal{L}_i\theta\circ\mathsf{V}(x)+
\frac{m}{\beta}	\theta\circ\mathsf{V}(x)\,\mathcal{Q}(x) \gamma_i\\
	&\,\le\,\frac{m c_i}{\beta}\theta\circ\mathsf{V}(x)+\frac{m\gamma_i}{\beta}\mathcal{L}_i\theta\circ\mathsf{V}(x)+\frac{m}{\beta}\theta\circ\mathsf{V}(x)\left(\sum_{j\in\mathbb{S}}\mathrm{q}_{ij}\gamma_j+\gamma_i\mathrm{o}(1)\right)\\
	&\,=\,\left(c_i+\gamma_i\frac{\mathcal{L}_i\theta\circ\mathsf{V}(x)}{\theta\circ\mathsf{V}(x)}-c_i-\beta+\gamma_i\mathrm{o}(1)\right)\frac{m}{\beta}\theta\circ\mathsf{V}(x)\\
	&\,=\,\left(\gamma_i\frac{\mathcal{L}_i\theta\circ\mathsf{V}(x)}{\theta\circ\mathsf{V}(x)}-\beta+\gamma_i\mathrm{o}(1)\right)\frac{m}{\beta}\theta\circ\mathsf{V}(x)\,,
	\end{align*}
	where in the second line we use assumption (ii).
	By assumption, for all $|x|$ large enough it holds that $$\sup_{i\in\mathbb{S}}\left|\gamma_i\frac{\mathcal{L}_i\theta\circ\mathsf{V}(x)}{\theta\circ\mathsf{V}(x)}+\gamma_i\mathrm{o}(1)\right|\,<\,\frac{\beta}{m}\,.$$
	Thus, for all $i\in\mathbb{S}$ and all $|x|$ large enough, $$\mathcal{L}\mathcal{V}(x,i)\,\le\,-(m-1)\,\theta\circ\mathsf{V}(x)\,\le\,-(m-1)\,\theta\bigl(\mathcal{V}(x,i)/(m-1)\bigr)\,\le\, -\theta\circ\mathcal{V}(x,i)\,,$$ where in the last step we employed the subadditivity property of $\theta(u)$. The assertion of the theorem now follows from 
	\cite[Theorems 5.1 and 7.1]{Tweedie-1994} (which ensure that every compact set is petite set for $\{(\X(x,i;t),\Lambda(x,i;t))\}_{t\ge0}$) and \cite[Theorems 3.2 and 3.4]{Douc-Fort-Guilin-2009}.
	\end{proof}

In the next proposition we show that if $\theta(u)$ is linear, then $\{(\X(x,i;t),\Lambda(x,i;u))\}_{t\ge0}$ is geometrically ergodic.

\begin{proposition}\label{prop}
	Let $\{(\X(x,i;t),\Lambda(x,i;t))\}_{t\ge0}$
be an open-set irreducible and aperiodic regime-switching diffusion process satisfying 
\textbf{(A1)}-\textbf{(A5)}. Assume 

\medskip 

\begin{itemize}
	\item [(i)] there are $\{c_i\}_{i\in\mathbb{S}}\subset\R$ and twice continuously differentiable $\mathsf{V}:\R^d\to(1,\infty)$, 
	such that
	\begin{align*} \limsup_{|x|\to\infty}\frac{\mathcal{L}_i\mathsf{V}(x)}{\mathsf{V}(x)}\,<\,c_i\,,\end{align*}

	\medskip
	
	\item[(ii)] $\mathcal{Q}(x)=\mathcal{Q}+\mathrm{o}(1)$ as $|x|\to\infty$,
\end{itemize}

\medskip

\noindent and either one of the following two conditions

\medskip

\begin{itemize}
	\item[(iii)] 
	$\mathcal{Q}=(\mathrm{q}_{ij})_{i,j\in\mathbb{S}}$ is the infinitesimal generator of an irreducible right-continuous temporally-homogeneous Markov chain  on $\mathbb{S}$ with invariant probability measure by $\uplambda=(\uplambda_i)_{i\in\mathbb{S}}$ and $\displaystyle\sum_{i\in\mathbb{S}}c_i\uplambda_i<0$.
	
	\medskip
	
	\item[(iii')] the matrix $-(\mathcal{Q}+\mathrm{diag}\,c)$ is a nonsingular $\mathcal{M}$-matrix, where $c=(c_i)_{i\in\mathbb{S}}.$
\end{itemize}

\medskip

\noindent
 Then,  $\{(\X(x,i;t),\Lambda(x,i;t))\}_{t\ge0}$
	is geometrically ergodic.
\end{proposition}
\begin{proof} 
	Assume first (i)-(iii). Analogously as in \cite[Theorem 2.1]{Shao-2015} we conclude that there are $\zeta\in(0,1)$, $\eta>0$ and $\gamma=(\gamma_i)_{i\in\mathbb{S}}$ with strictly positive components, such that $$(\mathcal{Q}+\zeta\,\mathrm{diag}\, c)\gamma\,=\,-\eta \,\gamma\,.$$
	Define $\mathcal{V}(x,i)\df \gamma_i \mathsf{V}^\zeta(x).$ Since $\zeta\in(0,1)$, it is straightforward to check that $$\mathcal{L}_i\mathsf{V}^\zeta(x)\,\le\, \zeta\,\mathsf{V}^{\zeta-1}(x)\,\mathcal{L}_i\mathsf{V}(x)$$ for all $x\in\R^d$ and $i\in\mathbb{S}$. Thus, for all $i\in\mathbb{S}$ and $|x|$ large enough, we have
	\begin{equation}\label{geom}
	\begin{aligned}\mathcal{L}\mathcal{V}(x,i)&\,=\,\gamma_i\mathcal{L}_i\mathsf{V}^\zeta(x)+\mathsf{V}^\zeta(x)\,\mathcal{Q}(x) \gamma_i\\
	&\,\le\,\zeta\,c_i\gamma_i\mathsf{V}^\zeta(x)+\mathsf{V}^\zeta(x)\sum_{j\in\mathbb{S}}\mathrm{q}_{ij}\gamma_j+\gamma_i\mathsf{V}^\zeta(x)\mathrm{o}(1)\\
	&\,=\,-\eta\,\gamma_i\mathsf{V}^\zeta(x)+\gamma_i\mathsf{V}^\zeta(x)\mathrm{o}(1)\\
	&\,\le\,-\eta\, \mathcal{V}(x,i)+\mathcal{V}(x,i)\mathrm{o}(1)\,,
	\end{aligned}\end{equation}
	which is exactly 	the Lyapunov equation on 
\cite[page 529]{Meyn-Tweedie-AdvAP-III-1993}
with $c=\eta-\epsilon$ ($0<\epsilon<\eta$ is such that $|\mathrm{o}(1)|<\epsilon$ for all $|x|$ large enough), $f(x,i)=\mathcal{V}(x,i)$, $C=\overline{\mathscr{B}}_{r}(0)\times\mathbb{S}$  for $r>0$ large enough and $d=\sup_{(x,i)\in C}|\mathcal{L}\mathcal{V}(x,i)|$. According to  	\cite[Theorems 5.1 and 7.1]{Tweedie-1994}, together with open-set irreducibility and $\mathcal{C}_b$-Feller property of $\{(\X(x,i;t),\Lambda(x,i;t))\}_{t\ge0}$, it follows that $C$ is a petite set for $\{(\X(x,i;t),\Lambda(x,i;t))\}_{t\ge0}$.
Consequently, from 
\cite[Proposition 6.1]{Meyn-Tweedie-AdvAP-II-1993}, \cite[Theorem 4.2]{Meyn-Tweedie-AdvAP-III-1993} and 
aperiodicity it follows now that there are a petite set $\mathcal{C}\in\mathfrak{B}(\R^d)\times\mathscr{P}(\mathbb{S})$, 	 $T>0$ and a non-trivial measure $\upnu_\mathcal{C}$ on $\mathfrak{B}(\R^d)\times\mathscr{P}(\mathbb{S})$, such that $\upnu_\mathcal{C}(\mathcal{C})>0$ and $$\p(t,(x,i),B)\,\ge\,\upnu_\mathcal{C}(B)$$ for all $(x,i)\in \mathcal{C}$, $t\ge T$ and $B\in\mathfrak{B}(\R^d)\times\mathscr{P}(\mathbb{S}).$ In particular, 
$$\p(t,(x,i),\mathcal{C})\,>\,0$$ for all $(x,i)\in \mathcal{C}$ and $t\ge T$,	
which is exactly the definition of aperiodicity used on \cite[page 1675]{Down-Meyn-Tweedie-1995}. Finally, observe that \cref{geom} is also  the Lyapunov equation used on   \cite[page 1679]{Down-Meyn-Tweedie-1995} with $c=\eta-\epsilon$ ($\epsilon$ is defined as above), $\tilde V(x,i)=\mathcal{V}(x,i)$,   $C=\overline{\mathscr{B}}_{r}(0)\times\mathbb{S}$  for $r>0$ large enough and $b=\sup_{(x,i)\in C}|\mathcal{L}\mathcal{V}(x,i)|$. The assertion now follows from \cite[Theorem 5.2]{Down-Meyn-Tweedie-1995}.

	Assume now (i), (ii) and (iii'). 
	Since $-(\mathcal{Q}+\mathrm{diag}\,c)$ is a nonsingular $\mathcal{M}$-matrix, there is a vector $\gamma=(\gamma_i)_{i\in\mathbb{S}}$ with strictly positive components such that the vector $$\delta=(\delta_i)_{i\in\mathbb{S}}=-(\mathcal{Q}+\mathrm{diag}\,c)\gamma$$ also hast strictly positive components. 
	Define 
	$\mathcal{V}(x,i)\df\gamma_i\mathsf{V}(x)$ and $\beta\df\inf_{i\in\mathbb{S}}\delta_i>0$.  Analogously as above we see that for all $i\in\mathbb{S}$ and all $|x|$ large enough,
	$$\mathcal{L}\mathcal{V}(x,i)\,\le\,-\beta\,\mathcal{V}(x,i)+\mathcal{V}(x,i)\mathrm{o}(1)\,,$$  which concludes the proof.
	\end{proof}

In the sequel, we discuss an example satisfying conditions from \Cref{tm:TV}.

\begin{example}
	Let   $\mathbb{S}=\{0,1\}$,
and let  $\mathrm{q}_{01}=\mathrm{q}_{10}=1$. Hence, $\uplambda=(1/2,1/2)$.
Further, 	
	  let
$$\bb(x,i) \,=\,\left\{ 
\begin{array}{ll} 
b\,, & i=0\,, \\ 
-{\rm sgn}(x)\beta(x)\,, & i=1\,,
\end{array} \right.$$	with $b\in\R$ and $\beta:\R\to[0,\infty)$ satisfying 	\begin{align}\label{eq:ex_TV_g}
\lim_{|x| \to \infty} |x|^{1-2p} \beta(x) \,=\, \infty
\end{align} for some $p\in[1/2,1)$, and	let 
$\upsigma(x,i)\equiv\upsigma(i)>0$ (implying open-set irreducibility and aperiodicity of the process). 
	Define  $\mathsf{V} : \R \to (1,\infty)$ by $\mathsf{V}(x) \df 1 + x^2$, and $\theta : (1,\infty)\to(0,\infty)$ by $\theta(u) \df  u^p$. Clearly,
	\begin{align*}\lim_{u\to \infty}\theta'(u)\,=\,0\qquad\text{and}\qquad  \lim_{|x|\to\infty}\frac{\theta\circ\mathsf{V}(x)}{\mathsf{V}(x)}&\,=\,0\,.\end{align*}
	Further,
	$$\mathcal{L}_i\mathsf{V}(x) \,=\,\left\{ 
	\begin{array}{ll} 
	2bx+\upsigma(0)^2\,, & i=0\,, \\ 
	-2{\rm sgn}(x)x\beta(x)+\upsigma(1)^2\,, & i=1\,.
	\end{array} \right.$$
	From this and \cref{eq:ex_TV_g} it follows that 
		$$\lim_{|x| \to \infty}\frac{\mathcal{L}_i\mathsf{V}(x)}{\theta\circ\mathsf{V}(x)} \,=\,\lim_{|x| \to \infty}\frac{\mathcal{L}_i\theta\circ\mathsf{V}(x)}{\theta\circ\mathsf{V}(x)} \,=\,\left\{ 
	\begin{array}{ll} 
	0\,, & i=0\,, \\ 
	-\infty\,, & i=1\,.
	\end{array} \right.$$
Hence, the process satisfies the conditions from \Cref{tm:TV} with arbitrary $c_0>0$ and $c_1=-2c_0,$ which implies subgeometric ergodicity with rate 	
	$ r(t) =  t^{p/(1-p)}.$
\end{example}


 \section{Ergodicity with respect to Wasserstein distances}\label{S4}
In this section, we discuss subgeometric ergodicity of $\{(\X(x,i;t),\Lambda(x,i;t))\}_{t\ge0}$
with respect to  a class of Wasserstein distances. In addition to \textbf{(A1)}-\textbf{(A5)}, throughout the section we assume 

\medskip

\begin{itemize}
	\item $\{\Lambda(x,i;t)\}_{t\ge0}$ and $\upsigma(x,i)$ are state independent, that is, $\Lambda(x,i;t)=\Lambda(i;t)$ and $\upsigma(x,i)=\upsigma(i)$  for all $(x,i)\in\R^d\times\mathbb{S}$ and $t\ge0$;
	
	\medskip
	
	\item $\{\Lambda(i;t)\}_{t\ge0}$ is irreducible.
\end{itemize}

 \medskip

We start with the following lemma.

\begin{lemma}\label{lemma:WASS}
	Let $i\in\mathbb{S}$, $\tau\ge0$,   $\{\Gamma_j\}_{j\in\mathbb{S}}\subset(-\infty,0]$,   $F:[0,T) \to [0, \infty)$ with $0<T\le\infty$, and $\psi:[0, +\infty) \to (0, \infty)$ be such that
	
	\medskip
	
	\begin{enumerate}[(i)]
		\item $F(t)$ is  absolutely continuous on $[t_0,t_1]$ for any $0<t_0<t_1<T$;
		
			\medskip
		
		\item $F'(t)\le \Gamma_{\Lambda(i;\tau+t)}\psi(F(t))$ a.e. on $[0,T)$;
		
			\medskip
		
		\item  $\Psi_{F(0)}(t) \df \int_t^{F(0)} \D s/\psi(s) < \infty$ for all $t \in (0, F(0)]$.
	\end{enumerate}

	\medskip

\noindent	Then 
	\begin{equation*}
	F(t) \,\leq \,\Psi_{F(0)}^{-1}\left(-\int_0^t \Gamma_{\Lambda(i;\tau+s)} \D s\right) 
	\end{equation*} for all $t\in[0,T)$ such that $ -\int_0^t \Gamma_{\Lambda(i;\tau+s)} \D s< \Psi_{F(0)}(0).$ In addition, if
	there is $\gamma\in[F(0),\infty]$ such that 
	$\Psi_{\gamma}(t) = \int_t^{\gamma} \D s/\psi(s) < \infty$ for all $t \in (0, \gamma]$, then \begin{equation*}
	F(t) \,\leq \,\Psi_{\gamma}^{-1}\left(-\int_0^t \Gamma_{\Lambda(i;\tau+s)} \D s\right) 
	\end{equation*} for all $t\in[0,T)$ such that $0\le -\int_0^t \Gamma_{\Lambda(i;\tau+s)} \D s< \Psi_{\gamma}(0).$
	 Furthermore, if $\psi(t)$ is convex and vanishes at zero, then $\Psi_{F(0)}(0)=\infty.$ In particular, the previous relations hold for all $t\in[0,T)$.
	\end{lemma}
\begin{proof}
	We have,
	\begin{equation*}
	-\Psi_{F(0)}(F(t))\,=\, \int_{F(0)}^{f(t)} \frac{\D s}{\psi(s)}\, =\, \int_0^t \frac{F'(s)\D s}{\psi(F(s))}\, \leq\, \int_0^t \Gamma_{\Lambda(i;\tau+s)} \D s
	\end{equation*} for all $t\in[0,T)$,
	which proves the first assertion.
The second claim follows from the fact that $\Psi_{F(0)}(t)\le\Psi_{\gamma}(t)$ for all  $t\in(0,F(0)]$, while
the last part follows from the convexity of $\psi(t)$: $$\psi(t)\,=\,\psi(t+(1-t)0)\,\le\,t\,\psi(1)+(1-t)\,\psi(0)\,=\,t\,\psi(1)$$ for all $t\in[0,1]$.
\end{proof}

We next prove \cref{eq23}.

\begin{lemma}\label{lemma:theta}
	Recall that $\{\Lambda(i;t)\}_{t\ge0}$ is  irreducible, 	$\{\bar\Lambda(i;t)\}_{t\ge0}$ is its independent copy, $$\tau_{ij}\,=\, \left\{ 
	\begin{array}{ll} 
	\inf\{t>0\colon \Lambda(i;t)\,=\,\bar\Lambda(j;t)\}\,, & i\neq j\,, \\ 
	0\,, & i=j\,,
	\end{array} \right. $$
	$\zeta = \inf_{i,j \in \mathbb{S}} \Prob(\Lambda(i;1)=j)$ and $ \vartheta= -\log (1-\zeta)$. Then, $$\Prob(\tau_{ij}>t)\,\le\,\E^{-\vartheta \lfloor t \rfloor}.$$
	for all $i,j\in\mathbb{S}$ and $t\ge0$.
\end{lemma}
\begin{proof}
	First, observe that for any initial distribution $\upmu=(\upmu_i)_{i\in\mathbb{S}}$ of $\{\Lambda(i;t)\}_{t\ge0}$ and all $j \in \mathbb{S}$ it holds that 
	\begin{equation}\label{eq:mc1}
	\sum_{i\in\mathbb{S}}\Prob(\Lambda(i;1) = j)\,\upmu_i \,\ge\, \zeta\,.
	\end{equation}
	Thus, for any initial distribution $\Pi=(\Pi_{i,j})_{i,j\in\mathbb{S}}$ of $\{(\Lambda(i;t),\bar\Lambda(j;t))\}_{t\ge0}$
	we have that	
	\begin{equation}
	\begin{aligned}\label{eq11}
	&\sum_{i,j\in\mathbb{S}}\Prob(\Lambda(i;1) \neq \bar\Lambda(j;1))\,\Pi_{i,j} \\&\, = \, \sum_{i,j\in\mathbb{S}}\sum_{k \in \mathbb{S}} \Prob(\Lambda(i,1) \neq \bar\Lambda(j;1),\, \Lambda(i;1) = k)\,\Pi_{i,j} \\
	&\, = \, \sum_{i,j\in\mathbb{S}}\sum_{k \in \mathbb{S}} \big( \Prob(\Lambda(i;1) \neq \bar\Lambda(j;1),\, \Lambda(i;1) = k) + \Prob(\Lambda(i;1) = \bar \Lambda(j;1),\, \Lambda(i;1) = k)\\&\hspace{2.1cm} - \Prob(\Lambda(i;1) = \bar\Lambda(j;1),\, \Lambda(i;1) = k) \big)\,\Pi_{i,j}\\
	&\,=\, \sum_{i,j\in\mathbb{S}}\sum_{k \in \mathbb{S}} \Prob(\Lambda(i;1) = k)\,\Pi_{i,j} - \sum_{i,j\in\mathbb{S}}\sum_{k \in \mathbb{S}} \Prob(\Lambda(i;1) = k,\, \Lambda(j;1) = k)\,\Pi_{i,j} \\
	& \,=\, 1 - \sum_{i,j\in\mathbb{S}}\sum_{k \in \mathbb{S}} \Prob(\Lambda(i;1) = k) \Prob (\bar\Lambda(j;1) = k)\,\Pi_{i,j} \\
	&\,\le\, 1 - \zeta \,.
	\end{aligned}
	\end{equation}
	Let
	$t\geq1$ be arbitrary. Then,  $t = \lfloor t \rfloor + s$ for some $s \in [0,1)$. We  have that
$$
	\Prob(\tau_{ij} > t) \,=\, \Prob(\tau_{ij} > t, \, \tau_{ij} > t-1)\,\leq\, \Prob(\Lambda(i;t) \neq \bar\Lambda(j;t),\, \tau_{ij} > t-1)\,.
	$$
	Observe that $\{ \tau_{ij} > t-1 \} \in \bar{\mathcal{F}}_{t-1}$. Here, $\{\bar{\mathcal{F}}_{t}\}_{t\ge0}$ stands for the natural filtration of the process $\{(\Lambda(i;t),\bar\Lambda(j;t))\}_{t\ge0}$. This and \cref{eq11} imply \begin{align*}&\Prob(\Lambda(i;t) \neq \bar\Lambda(j;t),\, \tau_{ij} > t-1)\\&\,=\,\int_{\{ \tau_{ij} > t-1 \}} \Prob(\Lambda(i;t) \neq \bar\Lambda(j;t) \mid \bar{\mathcal{F}}_{t-1})\,\D\mathbb{P}\\
	&\,=\,\int_{\{ \tau_{ij} > t-1 \}} \Prob\left(\Lambda(\Lambda(i;t-1);1) \neq \bar\Lambda(\bar\Lambda(j;t-1);1)\right)\D\mathbb{P}\\
	&\,\le\, (1-\zeta) \,\Prob( \tau_{ij} > t-1)\,.
	\end{align*}
	Thus, $$\Prob(\tau_{ij} > t) \,\le\,(1-\zeta)\, \Prob(\tau_{ij} > t-1)\,.$$ Iterating this procedure we arrive at 
\begin{equation*}
	\Prob(\tau_{ij} > t)\,\le\, (1-\zeta)^{\lfloor t \rfloor}\,=\,\E^{-\vartheta \lfloor t \rfloor}\end{equation*} for all $t\ge1$.
	However, it is clear that the  relation holds for all $t\ge0.$
\end{proof}

We now prove \Cref{tm:WASS-subgeom}.

\begin{proof}[Proof of Theorem \ref{tm:WASS-subgeom}]
Let 	$\{\bar\Lambda(i;t)\}_{t\ge0}$ be an independent copy of $\{\Lambda(i;t)\}_{t\ge0}$ (which is also independent of $\{\B(t)\}_{t\ge0}$). 
Define $$\tilde{\Lambda}(j;t)\,\df\, \left\{ 
\begin{array}{ll} 
\bar{\Lambda}(j;t)\,, & t<\tau_{ij}\,, \\ 
\Lambda(i;t)\,, & t\geq\tau_{ij}\,,
\end{array} \right. $$ for $t\ge0$.
By employing  Markov property, it is easy to see that $\{\tilde\Lambda(j;t)\}_{t\ge0}$
is a Markov chain with the same law as $\{\Lambda(j;t)\}_{t\ge0}$.

	Fix now $(x,i), (y,j)\in\R^d\times\mathbb{S}$, and let $\{(\X(x,i;t),\Lambda(i;t))\}_{t\ge0}$ and $\{(\X(y,j;t),\tilde\Lambda(j;t))\}_{t\ge0}$ be the corresponding solutions to \cref{eq1}.  Define 
$$
	\tau \,\df\, \inf\{ t>0 \colon (\X(x,i;t),\Lambda(i;t)) = (\X(y,j;t),\tilde\Lambda(j;t))\}\,.
	$$ Clearly, $\tau\ge\tau_{ij}.$
	Put
	$$\mathrm{Y}(y,j;t)\,\df\, \left\{ 
	\begin{array}{ll} 
	\X(y,j;t)\,, & t<\tau\,, \\ 
	\X(x,i;t)\,, & t\geq\tau\,,
	\end{array} \right. $$ for $t\ge0$.
	The marginals of the process  
	$\{(\mathrm{Y}(y,j;t),\tilde\Lambda(j;t))\}_{t\ge0}$ have the same law as the marginals of $\{(\X(y,j;t),\tilde\Lambda(j;t))\}_{t\ge0}$. Namely, we have that
		\begin{align*}
	&\mathbb{P} \bigl((\mathrm{Y}(y,j;t),\tilde\Lambda(j;t))\in B \times \{ k \}\bigr)  \\
	&\,=\,  \mathbb{P} \bigl((\mathrm{Y}(y,j;t),\tilde\Lambda(j;t))\in B \times \{ k \},\, t < \tau\bigr) +  \mathbb{P} \bigl((\mathrm{Y}(y,j;t),\tilde\Lambda(j;t))\in B \times \{ k \},\, t \geq \tau\bigr) \\
	&\,=\,  \mathbb{P} \bigl((\X(y,j;t),\tilde\Lambda(j;t))\in B \times \{ k \}, t < \tau\bigr) + \mathbb{P} \bigl((\X(x,i;t),\Lambda(i;t))\in B \times \{ k \},\, t \geq\tau\bigr)\,.
	\end{align*}
Further, by the strong Markov property we have that
		\begin{align*}
	&\mathbb{P} \bigl((\X(x,i;t),\Lambda(i;t))\in B \times \{ k \},\, t \geq\tau\bigr)\\
	&\,=\, \mathbb{E} \left[ \mathbb{E} \left[\mathbbm{1}_{\{t-\tau+\tau \geq \tau\}}\, \mathbbm{1}_{B \times \{ k \}} ((\X(x,i;t-\tau+\tau),\Lambda(i;t-\tau+\tau))  \mid \mathcal{F}_{\tau} \right] \right] \\
	&\,=\,  \mathbb{E}\left[ \mathbbm{1}_{\{t \geq \tau\}} \mathbb{P} \bigl( (\X(\X(x,i;\tau),\Lambda(i;\tau);t-\tau),\Lambda(\Lambda(i;\tau);t-\tau))\in B \times \{ k \},\, t-\tau\ge0\bigr) \right] \\
	&\,=\,  \mathbb{E}\left[ \mathbbm{1}_{\{t \geq \tau\}} \mathbb{P} \bigl( (\X(\X(y,i;\tau),\tilde\Lambda(i;\tau);t-\tau),\tilde\Lambda(\tilde\Lambda(i;\tau);t-\tau))\in B \times \{ k \},\, t-\tau\ge0\bigr) \right]\\
	&\,=\,\mathbb{P} \bigl((\X(y,i;t),\tilde\Lambda(i;t))\in B \times \{ k \},\, t \geq\tau\bigr)\,,
	\end{align*} which proves the assertion.
	Consequently,
\begin{align*} &\mathcal{W}^p_{f,p}\bigl(\updelta_{(x,i)}\mathcal{P}_t, \updelta_{(y,j)}\mathcal{P}_t\bigr)\, \leq\, \mathbb{E}\left[ \uprho\bigl((\X(x,i;t),\Lambda(i;t)),(\mathrm{Y}(y,j;t),\tilde\Lambda(j;t))\bigr)^p \right]\,. \end{align*}
	Now, we have
	\begin{align*}\label{eq:WASS-fin-bdd3}
	 \mathsf{W}^p_{f,p}\bigl(\updelta_{(x,i)}\mathcal{P}_t, \updelta_{(y,j)}\mathcal{P}_t\bigr) 
	&\,\le\, \mathbb{E}\left[ \uprho\bigl((\X(x,i;t),\Lambda(i;t)),(\mathrm{Y}(y,j;t)\tilde\Lambda(j;t))\bigr)^p \mathbbm{1}_{\{ \tau_{ij} > t/2 \}} \right] \\
	&\ \ \ \ + \mathbb{E}\left[ \uprho\bigl((\X(x,i;t),\Lambda(i;t)),(\mathrm{Y}(y,j;t)\tilde\Lambda(j;t))\bigr) ^p\mathbbm{1}_{\{ \tau_{ij} \leq t/2 \}} \right]\\
	&\,\leq\, (1+\gamma)^p \E^{-\vartheta \lfloor t/2 \rfloor} + \mathbb{E}\left[ f\bigl(|\X(x,i;t)-\mathrm{Y}(y,j;t)|\bigr)^p\, \mathbbm{1}_{\{ \tau_{ij} \leq t/2 \}}\right]\,,
	\end{align*} where in the last step  $\gamma\df\sup_{t>0}f(t)$ and we employed \Cref{lemma:theta}.
		After time $\tau_{ij}$ the processes $\{\Lambda(i;t)\}_{t\ge0}$ and $\{\tilde\Lambda(j;t)\}_{t\ge0}$ move together. Hence, by  \cref{eq:WASS-fin-bdd1}  it holds that (here we also use the fact that $\upsigma(x,i)=\upsigma(i)$)
	\begin{equation*}
	\frac{\D}{\D t} f\bigl(|\X(x,i;t)-\mathrm{Y}(y,j;t)|\bigr)\, \leq\, 0
	\end{equation*}
	a.e. on $[\tau_{ij},+\infty)$. Thus, $t\mapsto f(|\X(x,i;t)-\mathrm{Y}(y,j;t)|)$ is non-increasing on $[\tau_{ij},\infty)$ and 
	\begin{align*}
		& \mathcal{W}_{f,p}^p\bigl(\updelta_{(x,i)}\mathcal{P}_t, \updelta_{(y,j)}\mathcal{P}_t\bigr) \\
		&\,\le\,(1+\gamma)^p \E^{-\vartheta \lfloor t/2 \rfloor} + \mathbb{E}\left[ f\bigl(|\X(x,i;\tau_{ij}+t/2)-\mathrm{Y}(y,j;\tau_{ij}+t/2)|\bigr)^p\, \mathbbm{1}_{\{ \tau_{ij} \leq t/2 \}}\right]\,.
	\end{align*}
Define $$F(t)\,\df\,  f\bigl(|\X(x,i;\tau_{ij}+t)-\mathrm{Y}(y,j;\tau_{ij}+t)|\bigr)$$ for $t\ge0$.
By employing \cref{eq:WASS-fin-bdd1}   again, it follows that
	\begin{equation*} \label{eq:dft} \frac{\D }{\D t} F(t)\, \leq\, \Gamma_{\Lambda(i;\tau_{ij}+t)} \psi\bigl(F(t)\bigr)
	\end{equation*}
	a.e. on $[0,\tau-\tau_{ij})$. 
 By \Cref{lemma:WASS} we have that
	$$ F(t)\, \leq\, \Psi_{\gamma}^{-1} \left(- \int_{0}^{t} \Gamma_{\Lambda(i;\tau_{ij}+s)} \D s\right)\,=\, \Psi_{\gamma}^{-1} \left(- \int_{\tau_{ij}}^{t+\tau_{ij}} \Gamma_{\Lambda(i;s)} \D s\right)$$ on $[0,\infty)$.
	For $t\ge\tau$ the term on the left-hand side vanishes, and the term on the right-hand side is well defined and strictly positive ($\psi(u)$ is convex and $\psi(u)=0$  if, and only if, $u=0$). Thus,  \begin{align*}&\mathcal{W}_{f,p}^p\bigl(\updelta_{(x,i)}\mathcal{P}_t, \updelta_{(y,j)}\mathcal{P}_t\bigr)\,\le\, (1+\gamma)^p \E^{-\vartheta \lfloor t/2\rfloor} + \mathbb{E}\left[\left(\Psi_{\gamma}^{-1} \left(- \int_{\tau_{ij}}^{t/2+\tau_{ij}} \Gamma_{\Lambda(i;s)} \D s\right)\right)^p\right]\,.\end{align*}
	Birkhoff ergodic theorem implies that $$\lim_{t\to\infty}\frac{2}{t}\int_{\tau_{ij}}^{t/2+\tau_{ij}} \Gamma_{\Lambda(i;s)}\D s=\sum_{i \in \mathbb{S}}\Gamma_i\lambda_i\,<\, 0\,, \qquad \mathbb{P}\text{-a.s.}$$ 
Hence, since $\psi(u)$ is convex and $\psi(u)=0$  if, and only if, $u=0$, $$\lim_{t \to +\infty}\Psi_{\gamma}^{-1} \left(- \int_{\tau_{ij}}^{t/2+\tau_{ij}} \Gamma_{\Lambda(i;s)} \D s\right)\,=\,0$$ $\mathbb{P}$-a.s. This, together with dominated convergence theorem, shows the first assertion.

Assume now that $\psi(u)=u^q$ for $q > 1$. Then,
	$$
	\Psi_\gamma(t) \,=\,  \frac{1}{q-1} \bigl( t^{1-q}-\gamma^{1-q} \bigr)\qquad\text{and}\qquad \Psi_\gamma^{-1} (t) \,=\, \bigl( \gamma^{1-q} + (q-1) t \bigr)^{1/(1-q)}\,.$$
By employing Birkhoff ergodic theorem and Fatou's lemma, we have
	\begin{align*}
&\liminf_{t \to +\infty}\frac{\mathbb{E}\left[\left(\Psi_{\gamma}^{-1} \left(- \int_{\tau_{ij}}^{t/2+\tau_{ij}} \Gamma_{\Lambda(i;s)} \D s\right)\right)^p\right]} {\left(  \frac{(1-q)t}{2}\sum_{i \in \mathbb{S}} \Gamma_i \uplambda_i\right)^{p/(1-q)}}\\
& \,=\,
\liminf_{t \to +\infty}\mathbb{E} \left[ \frac{\left( \gamma^{1-q} + (1-q)\int_{\tau_{ij}}^{t/2+\tau_{ij}} \Gamma_{\Lambda(i;s)} \D s \right)^{p/(1-q)}} {\left(  \frac{(1-q)t}{2}\sum_{i \in \mathbb{S}} \Gamma_i \uplambda_i\right)^{p/(1-q)}} \right]\\
&\, \ge\,
 1\,.
\end{align*}
On the other side, 
Birkhoff ergodic theorem, Fatou's lemma
and Jensen's inequality imply
	\begin{align*}
	&\liminf_{t \to +\infty}\frac{\left(  \frac{(1-q)t}{2}\sum_{i \in \mathbb{S}} \Gamma_i \uplambda_i\right)^{p/(1-q)}} {\mathbb{E}\left[\left(\Psi_{\gamma}^{-1} \left(- \int_{\tau_{ij}}^{t/2+\tau_{ij}} \Gamma_{\Lambda(i;s)} \D s\right)\right)^p\right]}\\
	& \,=\,
	\liminf_{t \to +\infty}\mathbb{E} \left[ \frac{\left( \gamma^{1-q} + (1-q)\int_{\tau_{ij}}^{t/2+\tau_{ij}} \Gamma_{\Lambda(i;s)} \D s \right)^{p/(1-q)}} {\left(  \frac{(1-q)t}{2}\sum_{i \in \mathbb{S}} \Gamma_i \uplambda_i\right)^{p/(1-q)}} \right]^{-1}\\
	&\, \ge\,
	\liminf_{t\to\infty} \mathbb{E} \left[ \frac{\left( \gamma^{q-1} + (1-q)\int_{\tau_{ij}}^{t/2+\tau_{ij}} \Gamma_{\Lambda(i;s)} \D s \right)^{p/(1-q)}} {\left(  \frac{(1-q)t}{2}\sum_{i \in \mathbb{S}} \Gamma_i \uplambda_i\right)^{p/(q-1)}} \right]\\ 
	&\,=\, 1\,.
	\end{align*}
	Thus, $$\lim_{t \to +\infty}\frac{\mathbb{E}\left[\left(\Psi_{\gamma}^{-1} \left(- \int_{\tau_{ij}}^{t/2+\tau_{ij}} \Gamma_{\Lambda(i;s)} \D s\right)\right)^p\right]} {\left(  \frac{(1-q)t}{2}\sum_{i \in \mathbb{S}} \Gamma_i \uplambda_i\right)^{p/(1-q)}}\,=\,1\,,$$
		which proves the second relation.

	If $\psi(u)=\kappa u$ for $\kappa > 0$, then
	$$
	\Psi_\gamma(t) \,=\,  \frac{1}{\kappa} \ln\left(\frac{\gamma}{t}\right)\qquad\text{and}\qquad \Psi_\gamma^{-1} (t) \,=\, \gamma\,\E^{-\kappa t}\,.$$	
	Arguing as above, it holds that for any $0<\beta<-\kappa \sum_{i \in \mathbb{S}} \Gamma_i \uplambda_i$,
	$$\lim_{t \to +\infty}\frac{\mathbb{E}\left[\left(\Psi_{\gamma}^{-1} \left(- \int_{\tau_{ij}}^{t/2+\tau_{ij}} \Gamma_{\Lambda(i;s)} \D s\right)\right)^p\right]} {\E^{-p\beta t/2}}\,=\,0\,.$$
Hence, by taking $0<\alpha<\min\{\vartheta/p,-\kappa \sum_{i \in \mathbb{S}} \Gamma_i \uplambda_i\}$ the assertion follows.
		\end{proof}

We next prove \Cref{tm:WASS-subgeom1.1}.  First, we discuss conditions under which the process $\{\X(x,i;t)\}_{t\ge0}$ has  second moment.

\begin{lemma}\label{lemma:X_square}
	Assume \cref{linear}.
	Then, $$ \mathbb{E} \bigl[ |\X(x,i;t)|^2\bigr] \,\leq\  \bigl(1 +|x|^2\bigr) \E^{Kt}\,.$$
		Furthermore,  for any $\{\mathcal{F}_t\}_{t\ge0}$-stopping time $\tau$ such that $\mathbb{E}\bigl[\E^{2K\tau}\bigr]<\infty$ it follows that
		$$ \mathbb{E} \left[ |\X(x,i;\tau)|^2 \right] \,\le\, |x|^2 +  \bigl(1+|x|^2\bigr)\mathbb{E}\left[\E^{K\tau}\right]
		+ 4 \sup_{j\in\mathbb{S}}\mathrm{Tr}\bigl(\upsigma(j)\upsigma(j)^T\bigr)^{\frac{1}{2}} \bigl(1+|x|^2\bigr)^{\frac{1}{2}}\mathbb{E}\left[  \E^{2K  \tau} \right]^{\frac{1}{2}}\,.$$
\end{lemma}
\begin{proof}
Recall that $\{\X(x,i;t)\}_{t\ge0}$ satisfies	
	$$ \X(x,i;t) = x + \int_0^{t} \bb\bigl(\X(x,i;s),\Lambda(i;s)\bigr)\D s + \int_{0}^{t} \upsigma\bigl(\Lambda(i;s)\bigr)\D \B(s)\,.$$
		For $n\in\N$, define 
		$$\tau_n\,\df\, \inf \{ t\geq 0 : |\X(x,i;t)| \geq n \}\,.$$ 
By employing It\^{o}'s formula we conclude that
			\begin{align*}
			&|\X(x,i;t  \wedge \tau_n)|^2 \\&\,=\, |x|^2 + 2 \int_0^{t  \wedge \tau_n} \bigl\langle \X(x,i;s), \bb\bigl(\X(x,i;s), \Lambda(i;s)\bigr) \bigr\rangle \D s + \int_0^{t  \wedge \tau_n} \mathrm{Tr} \bigl(\upsigma\bigl(\Lambda(i;s)\bigr) \upsigma\bigl(\Lambda(i;s)\bigr)^T\bigr) \D s  \\
			 &\ \ \ \ \ \  + 2\int_0^{t  \wedge \tau_n} \X(x,i;s)^T \upsigma\bigl(\Lambda(x,i;s)\bigr) \D\B(s) \\
			&\,\leq\, |x|^2 + K\int_0^t \bigl( 1 + |\X(x,i;s)|^2\bigr) \mathbbm{1}_{[0, \tau_n]}(s)\, \D s +  2 \int_0^{t\wedge \tau_n} \X(x,i;s)^T \upsigma\bigl(\Lambda(x,i;s)\bigr) \D\B(s)  \\
			&\leq |x|^2 + K\int_0^t \bigl( 1 + |\X(x,i;s\wedge  \tau_n)|^2\bigr)  \D s +  2 \int_0^{t\wedge \tau_n} \X(x,i;s)^T \upsigma\bigl(\Lambda(x,i;s)\bigr)  \D\B(s)\,.
		\end{align*}
Since the last term on the right side is a martingale, we conclude that 
		$$  1+\mathbb{E} \left[ |\X(x,i;t  \wedge \tau_n)|^2 \right]\, \leq\,  1 + |x|^2 + K\int_0^t \bigl( 1 + \mathbb{E} \left[ |\X(x,i;s \wedge \tau_n)|^2 \right] \bigr) \D s\,.$$
		The first assertion now follows by employing Gr\" onwall's inequality and Fatou's lemma (observe that since $\{\X(x,i;t)\}_{t\ge0}$ is nonexplosive, $\mathbb{P}(\lim_{n\to \infty}\tau_n=\infty)=1$).

Assume now that  $\tau$ is a stopping time such that $\mathbb{E}[\E^{2K\tau}]<\infty$. By employing  It\^{o}'s lemma again we have that
	$$
			|\X(x,i;t)|^2
			\,\leq\, |x|^2 + Kt + K\int_0^t |\X(x,i;s)|^2 \D s+2 \int_0^t \X(x,i;s)^T \upsigma\bigl(\Lambda(i;s)\bigr) \D \B(s)\,. 
$$
		Gr\" onwall's inequality then gives
			\begin{align*} |\X(x,i;t)|^2 &\,\leq\,  |x|^2 + Kt + 2 \int_0^t \X(x,i;s)^T \upsigma\bigl(\Lambda(i;s)\bigr) \D \B(s) \\
			&\ \ \ \ \ + \int_0^t K \left( |x|^2 + Ks + 2 \int_0^s \X(x,i;u)^T \upsigma\bigl(\Lambda(i;u)\bigr) \D \B(u) \right) \E^{K(t-s)} \D s\,.
		\end{align*}
Consequently,  
		\begin{align*} &|\X(x,i;t \wedge \tau)|^2\\&\,\leq\,  |x|^2 + K(t \wedge \tau) + 2 \int_0^{t \wedge \tau} \X(x,i;s)^T \upsigma\bigl(\Lambda(i;s)\bigr) \D \B(s) \\
		&\ \ \ \ \ + \int_0^{t \wedge \tau} K \left( |x|^2 + Ks + 2 \int_0^s \X(x,i;u)^T \upsigma\bigl(\Lambda(i;u)\bigr) \D \B(u) \right) \E^{K(t \wedge \tau-s)} \D s\\
		&\,=\,|x|^2 + K(t \wedge \tau ) + 2 \int_0^{t \wedge \tau} \X(x,i;s)^T \upsigma\bigl(\Lambda(i;s)\bigr) \D \B(s) \\ &\ \ \ \ \ + |x|^2\bigl(\E^{Kt \wedge \tau}-1\bigr)+\E^{Kt \wedge \tau}-K(t \wedge \tau)\\&\ \ \ \ \ +2K
			\int_0^{t}  \left( \mathbbm{1}_{[0,t \wedge \tau]}(s)\, \E^{K(t \wedge \tau-s)}\int_0^s \X(x,i;u)^T \upsigma\bigl(\Lambda(i;u)\bigr) \D \B(u) \right)  \D s\\
			&\,\le\, |x|^2 + 2 \int_0^{t \wedge \tau} \X(x,i;s)^T \upsigma\bigl(\Lambda(i;s)\bigr) \D \B(s) + \bigl(1+|x|^2\bigr)\E^{K\tau}\\&\ \ \ \ \ +2K
			\int_0^{t}  \left( \mathbbm{1}_{[0,t \wedge \tau]}(s)\, \E^{K(t \wedge \tau-s)}\int_0^s \X(x,i;u)^T \upsigma\bigl(\Lambda(i;u)\bigr) \D \B(u) \right)  \D s\,.
		\end{align*}
		Taking expectation (and  the previous result) we have that
		\begin{align*} &\mathbb{E}\bigl[|\X(x,i;t \wedge \tau)|^2\bigr]\\&\,\leq\,
		|x|^2 +  \bigl(1+|x|^2\bigr)\mathbb{E}\left[\E^{K\tau}\right]\\&\ \ \ \ \ +2K
		\int_0^{t}  \mathbb{E}\left[ \mathbbm{1}_{[0,t \wedge \tau]}(s)\, \E^{K(t \wedge \tau-s)}\int_0^s \X(x,i;u)^T \upsigma\bigl(\Lambda(i;u)\bigr) \D \B(u) \right]  \D s
		\\&\,\leq\,
		|x|^2 +  \bigl(1+|x|^2\bigr)\mathbb{E}\left[\E^{K\tau}\right]\\&\ \ \ \ \
		+2K \int_0^t \mathbb{E}\left[ \mathbbm{1}_{[0, t \wedge \tau]}(s)\, \E^{2K (t \wedge \tau-s)}  \right]^{1/2}  \mathbb{E}\left[ \left( \int_0^s \mathbbm{1}_{[0,  t \wedge \tau]} \X(x,i;u)^T \upsigma(\Lambda(i;u)) \D \B(u) \right)^2 \right]^{1/2}\D s \\
		\\&\,\leq\,
		|x|^2 +  \bigl(1+|x|^2\bigr)\mathbb{E}\left[\E^{K\tau}\right]
				+ 2K \sup_{j\in\mathbb{S}}\mathrm{Tr}\bigl(\upsigma(j)\upsigma(j)^T\bigr)^{1/2}\mathbb{E}\left[  \E^{2K  \tau} \right]^{1/2}\\&\ \ \ \ \	\int_0^t \E^{-Ks}  \mathbb{E}\left[ \int_0^s \mathbbm{1}_{[0,  t \wedge \tau]}(s)\, |\X(x,i;u)|^2 \D u \right]^{1/2} \D s \\
			&\,\leq\,|x|^2 +  \bigl(1+|x|^2\bigr)\mathbb{E}\left[\E^{K\tau}\right]
			+ 2K \sup_{j\in\mathbb{S}}\mathrm{Tr}\bigl(\upsigma(j)\upsigma(j)^T\bigr)^{1/2}\mathbb{E}\left[  \E^{2K  \tau} \right]^{1/2}\\&\ \ \ \ \
			 \int_0^\infty \E^{-Ks}  \mathbb{E}\left[ \int_0^s |\X(x,i;u)|^2 \D u \right]^{1/2} \D s \\
			&\,\leq\,|x|^2 +  \bigl(1+|x|^2\bigr)\mathbb{E}\left[\E^{K\tau}\right]
			+ 2K \sup_{j\in\mathbb{S}}\mathrm{Tr}\bigl(\upsigma(j)\upsigma(j)^T\bigr)^{1/2} \bigl(1+|x|^2\bigr)^{1/2}\mathbb{E}\left[  \E^{2K  \tau} \right]^{1/2} 
			\int_0^\infty \E^{-Ks/2}   \D s\\
			&\,=\,|x|^2 +  \bigl(1+|x|^2\bigr)\mathbb{E}\left[\E^{K\tau}\right]
			+ 4 \sup_{j\in\mathbb{S}}\mathrm{Tr}\bigl(\upsigma(j)\upsigma(j)^T\bigr)^{1/2} \bigl(1+|x|^2\bigr)^{1/2}\mathbb{E}\left[  \E^{2K  \tau} \right]^{1/2} \,,		
		\end{align*}
	where in the third step we used It\^{o}'s isometry and in the fifth step we used the first assertion of the lemma.
\end{proof}

 We are now ready to prove \Cref{tm:WASS-subgeom1.1}.

\begin{proof}[Proof of \Cref{tm:WASS-subgeom1.1}] 
By using the same reasoning (and notation) as in the proof of \Cref{tm:WASS-subgeom}, we have that	
	\begin{align*} &\mathcal{W}_{f,1}\bigl(\updelta_{(x,i)}\mathcal{P}_t, \updelta_{(y,j)}\mathcal{P}_t\bigr)\, \leq\, \mathbb{E}\left[ \uprho\bigl((\X(x,i;t),\Lambda(i;t)),(\mathrm{Y}(y,j;t),\tilde\Lambda(j;t))\bigr) \right]\,. \end{align*}
Further, for $\varepsilon>0$ such that $K+K\varepsilon<\vartheta$ (such $\varepsilon$ exists since by assumption $K<\vartheta$) it follows that
	\begin{align*}
	&\mathbb{E}\left[ \uprho\bigl((\X(x,i;t),\Lambda(i;t)),(\mathrm{Y}(y,j;t),\tilde\Lambda(j;t))\bigr) \right] \\
	&\,=\, \mathbb{E}\left[ \bigl(\mathbbm{1}_{\{ \Lambda(i;t)\neq \tilde\Lambda(j;t)\}}+f(|\X(x,i;t)-\mathrm{Y}(y,j;t)|)\bigr) \mathbbm{1}_{\{ \tau_{ij} > t/(1+\varepsilon) \}} \right] \\
	&\ \ \ \ \ + \mathbb{E}\left[ f(|\X(x,i;t)-\mathrm{Y}(y,j;t)|) \mathbbm{1}_{\{ \tau_{ij} \leq t/(1+\varepsilon) \}} \right] \\
	&\,\leq\,  \mathbb{E}\left[ \bigl(1+f(|\X(x,i;t)-\mathrm{Y}(y,j;t)|)\bigr)^{2} \right]^{1/2} \Prob(\tau_{ij} > t/(1+\varepsilon))^{1/2}\\ &\ \ \ \ \  + \mathbb{E}\left[ f(|\X(x,i;t)-\mathrm{Y}(y,j;t)|) \mathbbm{1}_{\{ \tau_{ij} \leq t/(1+\varepsilon) \}}\right] \\
	&\,\leq\,  \mathbb{E}\left[ \bigl(1+f(|\X(x,i;t)|)+f(|\mathrm{Y}(y,j;t)|)\bigr)^{2} \right]^{1/2}\E^{-(\vartheta/2)\lfloor t/(1+\varepsilon)\rfloor}\\&\ \ \ \ \   + \mathbb{E}\left[ f(|\X(x,i;t)-\mathrm{Y}(y,j;t)|) \mathbbm{1}_{\{ \tau_{ij} \leq t/(1+\varepsilon) \}}\right] \\
	&\,\leq\, 3^{1/2} \left(1+\mathbb{E}\left[ f(|\X(x,i;t)|)^2\right] + \mathbb{E}\left[ f(|\mathrm{Y}(y,j;t)|)^2\right] \right)^{1/2}\E^{-(\vartheta/2)\lfloor t/(1+\varepsilon)\rfloor}\\
	&\ \ \ \ \   + \mathbb{E}\left[ f(|\X(x,i;t)-\mathrm{Y}(y,j;t)|) \mathbbm{1}_{\{ \tau_{ij} \leq t/(1+\varepsilon) \}}\right]\\
	&\,\leq\, C_1\left(1+\mathbb{E}\left[ |\X(x,i;t)|^2\right] + \mathbb{E}\left[ |\mathrm{Y}(y,j;t)|^2\right] \right)^{1/2}\E^{-(\vartheta/2)\lfloor t/(1+\varepsilon)\rfloor}\\&\ \ \ \ \   + \mathbb{E}\left[ f(|\X(x,i;t)-\mathrm{Y}(y,j;t)|) \mathbbm{1}_{\{ \tau_{ij} \leq t/(1+\varepsilon) \}}\right]\\
	&\,\le\, C_2\left(1+|x|^2+|y|^2 \right)^{1/2}\E^{Kt/2-(\vartheta/2)\lfloor t/(1+\varepsilon)\rfloor}  + \mathbb{E}\left[ f(|\X(x,i;t)-\mathrm{Y}(y,j;t)|) \mathbbm{1}_{\{ \tau_{ij} \leq t/(1+\varepsilon) \}}\right]\,,
	\end{align*}
	for some $C_1,C_2>0$. Here, in the third step we used subadditivity property of concave functions and \Cref{lemma:theta}, in the fifth step we used the fact that $f(u)\le Au +B$ for some $A,B>0$ ($f(u)$ is concave), and in the last step we used \Cref{lemma:X_square}. Clearly, the first term on the right-hand side will converge to zero (as $t$ goes to infinity) due to the choice of $\varepsilon>0$.

We next discuss the second term. 	Analogously as in the proof of \Cref{tm:WASS-subgeom},  by employing  \cref{eq:WASS-fin-unbdd1},  it holds that 
\begin{equation*}
\frac{\D}{\D t} f\bigl(|\X(x,i;t)-\mathrm{Y}(y,j;t)|\bigr)\, \leq\, 0
\end{equation*}
a.e. on $[\tau_{ij},+\infty)$. Thus, $t\mapsto f(|\X(x,i;t)-\mathrm{Y}(y,j;t)|)$ is non-increasing on $[\tau_{ij},\infty)$ and 
\begin{align*}
\mathcal{W}_{f,1}\bigl(\updelta_{(x,i)}\mathcal{P}_t, \updelta_{(y,j)}\mathcal{P}_t\bigr) &\,\le\,C_2\left(1+|x|^2+|y|^2 \right)^{1/2}\E^{Kt/2-(\vartheta/2)\lfloor t/(1+\varepsilon)\rfloor} \\&\ \ \ \ \  +\mathbb{E}\left[F(\varepsilon t/(1+\varepsilon)) \mathbbm{1}_{\{ \tau_{ij} \leq t/(1+\varepsilon) \}}\right]\,,
\end{align*} where $F(t)$ is given (as in the proof of \Cref{tm:WASS-subgeom})
by $$F(t)\,=\,  f\bigl(|\X(x,i;\tau_{ij}+t)-\mathrm{Y}(y,j;\tau_{ij}+t)|\bigr)$$ for $t\ge0$.
We now have that
\begin{equation}
\begin{aligned}\label{eq:gen}
&\mathcal{W}_{f,1}\bigl(\updelta_{(x,i)}\mathcal{P}_t, \updelta_{(y,j)}\mathcal{P}_t\bigr)\\ &\,\le\,C_2\left(1+|x|^2+|y|^2 \right)^{1/2}\E^{Kt/2-(\vartheta/2)\lfloor t/(1+\varepsilon)\rfloor}\\ &\ \ \ \ \ +\mathbb{E}\left[F(\varepsilon t/(1+\varepsilon)) \mathbbm{1}_{\{ \tau_{ij} \leq t/(1+\varepsilon) \}}\mathbbm{1}_{\{f(|\X(x,i;\tau_{ij})-\mathrm{Y}(y,j;\tau_{ij})|)\le\eta \}}\right]\\
&\ \ \ \ \  +\mathbb{E}\left[F(\varepsilon t/(1+\varepsilon)) \mathbbm{1}_{\{ \tau_{ij} \leq t/(1+\varepsilon) \}}\mathbbm{1}_{\{f(|\X(x,i;\tau_{ij})-\mathrm{Y}(y,j;\tau_{ij})|)>\eta \}}\right]\,.
\end{aligned}
\end{equation}
On the event $\{f(|\X(x,i;\tau_{ij})-\mathrm{Y}(y,j;\tau_{ij})|)\le\eta \}$ we have the following. By employing \cref{eq:WASS-fin-unbdd1}   again, it follows that
\begin{equation*}\frac{\D }{\D t} F(t)\, \leq\, \Gamma_{\Lambda(i;\tau_{ij}+t)} \psi\bigl(F(t)\bigr)
\end{equation*}
a.e. on $[0,\tau-\tau_{ij})$. 
 \Cref{lemma:WASS} now implies that
$$ F(t)\, \leq\, \Psi_{f(|\X(x,i;\tau_{ij})-\mathrm{Y}(y,j;\tau_{ij})|)}^{-1} \left(- \int_{0}^{t} \Gamma_{\Lambda(i;\tau_{ij}+s)} \D s\right)\,\le\, \Psi_{\eta}^{-1} \left(- \int_{\tau_{ij}}^{t+\tau_{ij}} \Gamma_{\Lambda(i;s)} \D s\right)$$ on $[0,\infty)$.
For $t\ge\tau$ the term on the left-hand side vanishes, and the term on the right-hand side is well defined and strictly positive ($\psi(u)$ is convex and $\psi(u)=0$  if, and only if, $u=0$). Thus,  \begin{align*}&\mathbb{E}\left[F(\varepsilon t/(1+\varepsilon)) \mathbbm{1}_{\{ \tau_{ij} \leq t/(1+\varepsilon) \}}\mathbbm{1}_{\{f(|\X(x,i;\tau_{ij})-\mathrm{Y}(y,j;\tau_{ij})|)\le\eta \}}\right]\\&\,\le\,  \mathbb{E}\left[\Psi_{\eta}^{-1} \left(- \int_{\tau_{ij}}^{\varepsilon t/(1+\varepsilon)+\tau_{ij}} \Gamma_{\Lambda(i;s)} \D s\right)\mathbbm{1}_{\{f(|\X(x,i;\tau_{ij})-\mathrm{Y}(y,j;\tau_{ij})|)\le\eta \}}\right]\,.\end{align*}
Birkhoff ergodic theorem implies that $$\lim_{t\to\infty}\frac{1+\varepsilon}{\varepsilon t}\int_{\tau_{ij}}^{\varepsilon t/(1+\varepsilon)+\tau_{ij}} \Gamma_{\Lambda(i;s)}\D s=\sum_{i \in \mathbb{S}}\Gamma_i\lambda_i\,<\, 0$$ $\mathbb{P}\text{-a.s.}$ on $\{f(|\X(x,i;\tau_{ij})-\mathrm{Y}(y,j;\tau_{ij})|)\le\eta \}$. 
Hence, since $\psi(u)$ is convex and $\psi(u)=0$  if, and only if, $u=0$, $$\lim_{t \to \infty}\Psi_{\eta}^{-1} \left(- \int_{\tau_{ij}}^{\varepsilon t/(1+\varepsilon)+\tau_{ij}} \Gamma_{\Lambda(i;s)} \D s\right)\,=\,0$$ $\mathbb{P}$-a.s. This, together with dominated convergence theorem, shows that 
\begin{equation}
\label{eq:gen1}\lim_{t \to \infty}\mathbb{E}\left[F(\varepsilon t/(1+\varepsilon)) \mathbbm{1}_{\{ \tau_{ij} \leq t/(1+\varepsilon) \}}\mathbbm{1}_{\{f(|\X(x,i;\tau_{ij})-\mathrm{Y}(y,j;\tau_{ij})|)\le\eta \}}\right]\,=\,0\,.\end{equation}

On  the event $\{f(|\X(x,i;\tau_{ij})-\mathrm{Y}(y,j;\tau_{ij})|)>\eta \}$ we proceed as follows.
Recall that  $\delta = \inf \{ t\geq0 \colon f(1/t) \leq \eta \}$. It clearly must hold that $\delta>0$. Thus, since for $x,y\in\R^d$, $\lceil \delta |x-y| \rceil \geq \delta |x-y|$, we have that
	$$ f \left( \frac{|x-y|}{\lceil \delta |x-y| \rceil}\right) \,\leq\, f(1/\delta) \,\leq\, \eta\,.$$
Let $z_0, \dots, z_{\lceil \delta |\X(x,i;\tau_{ij})-\mathrm{Y}(y,j;\tau_{ij})| \rceil}  \in \R^d$ be such that $z_0 \df \X(x,i;\tau_{ij})$ and 
	$$ z_{k+1} \,\df\, z_k + \frac{\X(x,i;\tau_{ij})-\mathrm{Y}(y,j;\tau_{ij})}{\lceil \delta |\X(x,i;\tau_{ij})-\mathrm{Y}(y,j;\tau_{ij})| \rceil}\,,\qquad k=0,\dots,\lceil \delta |\X(x,i;\tau_{ij})-\mathrm{Y}(y,j;\tau_{ij})| \rceil-1\,.$$
	By definition, $z_{\lceil \delta |\X(x,i;\tau_{ij})-\mathrm{Y}(y,j;\tau_{ij})| \rceil}=\mathrm{Y}(y,j;\tau_{ij})$, $z_0, \dots, z_{\lceil \delta |\X(x,i;\tau_{ij})-\mathrm{Y}(y,j;\tau_{ij})| \rceil}$ are $\mathcal{F}_{\tau_{ij}}$-measurable and $$|z_{k+1}-z_k|\, \leq\, \frac{|\X(x,i;\tau_{ij})-\mathrm{Y}(y,j;\tau_{ij})|}{\lceil \delta |\X(x,i;\tau_{ij})-\mathrm{Y}(y,j;\tau_{ij})| \rceil}\,,$$ so $f(|z_{k+1}-z_k|) \leq \eta$ for $k=0,\ldots,\lceil \delta |\X(x,i;\tau_{ij})-\mathrm{Y}(y,j;\tau_{ij})| \rceil-1$. For $t\ge0$, let  $\tilde{\B}(t)\df\B(\tau_{ij}+t)-\B(t).$ Clearly, $\{\tilde{\B}(t)\}_{t\ge0}$ is a Brownian motion. Further, let $\{\tilde{\X}^{(\lceil \delta |\X(x,i;\tau_{ij})-\mathrm{Y}(y,j;\tau_{ij})| \rceil)}(t)\}_{t\ge0}=\{\mathrm{Y}(y,j;\tau_{ij}+t)\}_{t\ge0},$ and for  $k=0,\ldots,\lceil \delta |\X(x,i;\tau_{ij})-\mathrm{Y}(y,j;\tau_{ij})| \rceil-1$ let $\{\tilde{\X}^{(k)}(t)\}_{t\ge0}$ be solution to 
	\begin{align*}
		\D \tilde{\X}^{(k)}(t) &\,=\, \bb\bigl(\tilde{\X}^{(k)}(t),\Lambda(i;t)\bigr)\D t+\upsigma\bigl(\Lambda(i;t)\bigr)\D  \tilde{\B}(t)\\ \tilde{\X}^{(k)}(0)&\,=\, z_k \\
		\Lambda(i;0)&\,=\, i\in\mathbb{S}\,.
	\end{align*}
Observe that $\{\tilde{\X}^{(0)}(t)\}_{t\ge0}=\{\X(x,i;\tau_{ij}+t)\}_{t\ge0}.$
 We now have that
	$$F(t) \,\leq\, f\bigl(|\tilde{\X}^{(0)}(t)-\tilde{\X}^{(1)}(t)|\bigr) + \dots  + f\bigl(|\tilde{\X}^{(\lceil \delta |\X(x,i;\tau_{ij})-\mathrm{Y}(y,j;\tau_{ij})| \rceil-1)}(t)-\tilde{\X}^{(\lceil \delta |\X(x,i;\tau_{ij})-\mathrm{Y}(y,j;\tau_{ij})| \rceil)}(t)|\bigr)\,,$$ and from the first part of the proof it follows that 
	$$F(t)\,\leq\, \lceil \delta |\X(x,i;\tau_{ij})-\mathrm{Y}(y,j;\tau_{ij})| \rceil  \Psi_{\eta}^{-1} \left(- \int_{\tau_{ij}}^{t+\tau_{ij}} \Gamma_{\Lambda(i;s)} \D s\right)\,.
$$
By taking  expectation we get
	\begin{align*}
	&\mathbb{E}\left[F(\varepsilon t/(1+\varepsilon)) \mathbbm{1}_{\{ \tau_{ij} \leq t/(1+\varepsilon) \}}\mathbbm{1}_{\{f(|\X(x,i;\tau_{ij})-\mathrm{Y}(y,j;\tau_{ij})|)>\eta \}}\right]\\ &\, \leq\,
	\mathbb{E} \Bigg[\lceil \delta |\X(x,i;\tau_{ij})-\mathrm{Y}(y,j;\tau_{ij})| \rceil  \Psi_{\eta}^{-1} \left(- \int_{\tau_{ij}}^{\varepsilon t/(1+\varepsilon)+\tau_{ij}} \Gamma_{\Lambda(i;s)} \D s\right) \\ &\hspace{1.1cm} \mathbbm{1}_{\{ \tau_{ij} \leq t/(1+\varepsilon) \}}\mathbbm{1}_{\{f(|\X(x,i;\tau_{ij})-\mathrm{Y}(y,j;\tau_{ij})|)>\eta \}} \Bigg]\\
	&\,\leq\, \mathbb{E} \left[\lceil \delta |\X(x,i;\tau_{ij})-\mathrm{Y}(y,j;\tau_{ij})| \rceil^2\right]^{1/2}\\& \ \ \ \ \ \ \mathbb{E} \left[  \Psi_{\eta}^{-1} \left(- \int_{\tau_{ij}}^{\varepsilon t/(1+\varepsilon)+\tau_{ij}} \Gamma_{\Lambda(i;s)} \D s\right)^2 \mathbbm{1}_{\{ \tau_{ij}\, \leq\, t/(1+\varepsilon) \}}\mathbbm{1}_{\{f(|\X(x,i;\tau_{ij})-\mathrm{Y}(y,j;\tau_{ij})|)>\eta \}} \right]^{1/2}\,.
	\end{align*}
	From \Cref{lemma:X_square} we now that $$\mathbb{E} \left[\lceil \delta |\X(x,i;\tau_{ij})-\mathrm{Y}(y,j;\tau_{ij})| \rceil^2\right]\,=\,\mathbb{E} \left[\lceil \delta |\X(x,i;\tau_{ij})-\mathrm{X}(y,j;\tau_{ij})| \rceil^2\right]\,<\,\infty.$$ Thus, analogously as in \cref{eq:gen1} we have that \begin{equation}\label{eq:gen2}\lim_{t \to \infty}\mathbb{E}\left[F(\varepsilon t/(1+\varepsilon)) \mathbbm{1}_{\{ \tau_{ij} \leq t/(1+\varepsilon) \}}\mathbbm{1}_{\{f(|\X(x,i;\tau_{ij})-\mathrm{Y}(y,j;\tau_{ij})|)>\eta \}}\right]\,=\,0\,.
	\end{equation} Now, by combining \cref{eq:gen1,eq:gen2} the first assertion follows.
	
	The cases when $\psi(u)=u^q$ and $\psi(u)=\kappa u$ are treated in a completely the same way as in \Cref{tm:WASS-subgeom}.
\end{proof}

We continue with the  proof of  \Cref{tm:WASS-subgeom2}.

\begin{proof}[Proof of \Cref{tm:WASS-subgeom2}]
	Observe first that \cref{eq:subb} holds for any two initial distributions $\upmu$ and $\upnu$ of $\{(\X(x,i;t),\Lambda(i;t))\}_{t\ge0}$, that is,
	$$ \lim_{t\to\infty}\mathcal{W}_{f,p}\bigl(\upmu\mathcal{P}_t, \upnu\mathcal{P}_t\bigr)\, =\, 0\,.$$ From this we conclude that if $\{(\X(x,i;t),\Lambda(i;t))\}_{t\ge0}$ admits an invariant probability measure, then it must be unique. Namely, if $\uppi$ and $\overline{\uppi}$ were two invariant probability measures of\linebreak $\{(\X(x,i;t),\Lambda(i;t))\}_{t\ge0}$, then 
	$$ \mathcal{W}_{f,p}\bigl(\uppi, \overline\uppi\bigr)\,=\,\lim_{t\to\infty}\mathcal{W}_{f,p}\bigl(\uppi\mathcal{P}_t, \overline\uppi\mathcal{P}_t\bigr)\, =\, 0$$ which implies  $\uppi=\overline\uppi.$ 	
	Thus, if $\{(\X(x,i;t),\Lambda(i;t))\}_{t\ge0}$ admits an invariant probability measure $\uppi$, then
	$$ \lim_{t \to +\infty}\mathcal{W}_{f,p}\bigl( \updelta_{(x,i)}\mathcal{P}_t,\uppi\bigr)\,=\,\lim_{t\to\infty}\mathcal{W}_{f,p}\bigl(\updelta_{(x,i)}\mathcal{P}_t, \uppi\mathcal{P}_t\bigr)\, =\, 0\,.$$

	In the sequel we show that \cref{eq:subb} guarantees existence of an invariant probability measure of  $\{(\X(x,i;t),\Lambda(i;t))\}_{t\ge0}$.	
	According to \cite[Theorem 3.1]{Meyn-Tweedie-AdvAP-II-1993} this will follow if we show that for each $(x,i)\in\R^d\times\mathbb{S}$ and $0<\epsilon<1$ there is a compact set $C\subset\R^d$ (possibly depending on $(x,i)$ and $\epsilon$) such that $$\liminf_{t \to \infty}\frac{1}{t}\int_0^t\p(s,(x,i),C\times\mathbb{S})\,\D s\,\ge\,1-\epsilon\,.$$
	Let $r>0$ be large enough so that $$\inf_{x\in\mathscr{B}^c_r(0)}g(x)\,\ge\, -\inf_{x\in\mathscr{B}_r(0)}g(x)\,.$$ Such $r$ exists since $\lim_{|x|\to\infty}g(x)=\infty$. Observe that if the previous relation holds for some $r_0$, then it also holds for all $r\ge r_0$. We have that
	\begin{align*}
	\mathcal{L}\mathcal{V}(x,i)&\,\le\, - g(x)\mathbb{1}_{\mathscr{B}_r(0)}(x)-g(x)\mathbb{1}_{\mathscr{B}^c_r(0)}(x)\\
	&\,\le\,  \left(\bigl(\inf_{x\in\mathscr{B}^c_r(0)}g(x)\bigr)^{1/2}+ \inf_{x\in\mathscr{B}^c_r(0)}g(x)\right)\mathbb{1}_{\mathscr{B}_r(0)}(x)-\frac{1}{2}\inf_{x\in\mathscr{B}^c_r(0)}g(x)\mathbb{1}_{\mathscr{B}^c_r(0)}(x)\\
	&\,=\,  \left(\bigl(\inf_{x\in\mathscr{B}^c_r(0)}g(x)\bigr)^{1/2}+\frac{1}{2} \inf_{x\in\mathscr{B}^c_r(0)}g(x)\right)\mathbb{1}_{\mathscr{B}_r(0)\times\mathbb{S}}(x,i)-\frac{1}{2}\inf_{x\in\mathscr{B}^c_r(0)}g(x)\,.
	\end{align*} Now, according to \cite[Theorem 1.1]{Meyn-Tweedie-AdvAP-III-1993} we conclude that for each $(x,i)$ and $r$ large enough,
	$$\liminf_{t \to \infty}\frac{1}{t}\int_0^t\p\bigl(s,(x,i),\overline{\mathscr{B}}_r(0)\times\mathbb{S}\bigr)\,\D s\,\ge\,\frac{\frac{1}{2}\inf_{x\in\mathscr{B}^c_r(0)}g(x)}{\bigl(\inf_{x\in\mathscr{B}^c_r(0)}g(x)\bigr)^{1/2}+\frac{1}{2} \inf_{x\in\mathscr{B}^c_r(0)}g(x)}\,.$$
	The assertion now follows by choosing $r$ large enough. \end{proof}

Finally, we prove \Cref{tm:WASS-subgeom3}.

\begin{proof}[Proof of \Cref{tm:WASS-subgeom3}]
	In case (i), analogously as in the proof of \Cref{tm:TV} we conclude that there is a non-negative $\mathcal{V}\in\mathcal{C}^2(\R^d\times\mathbb{S})$  such that
	$$\lim_{|x|\to\infty}\inf_{i\in\mathbb{S}}\theta\circ\mathcal{V}(x,i)\,=\,\infty\qquad\text{and}\qquad
	\mathcal{L}\mathcal{V}(x,i)\,\le\, -\theta\circ\mathcal{V}(x,i)$$ for all $i\in\mathbb{S}$ and  $|x|$ large enough. In cases (ii) and (iii), by the same reasoning as in the proof of \Cref{prop} we see that there are $\eta>0$ and  non-negative $\mathcal{V}\in\mathcal{C}^2(\R^d\times\mathbb{S})$,  such that $$\lim_{|x|\to\infty}\inf_{i\in\mathbb{S}}\mathcal{V}(x,i)\,=\,\infty\qquad\text{and}\qquad
	\mathcal{L}\mathcal{V}(x,i)\,\le\, -(\eta-\epsilon)\,\mathcal{V}(x,i)$$ for  fixed $0<\epsilon<\eta$ and all $i\in\mathbb{S}$ and $|x|$ large enough. The desired result now follows by setting $f(x)\df\inf_{i\in\mathbb{S}}\theta\circ\mathcal{V}(x,i)$ in the first case, and $f(x)\df(\eta-\epsilon)\inf_{i\in\mathbb{S}}\mathcal{V}(x,i)$ in the second and third case.
\end{proof}

Typical examples satisfying conditions of \Cref{tm:WASS-subgeom,tm:WASS-subgeom1.1,tm:WASS-subgeom2,tm:WASS-subgeom3} are given as follows.

\begin{example}\label{ex1}{\rm 
		\begin{itemize}
			\item [(i)]
			Let   $\mathbb{S}=\{0,1\}$,  let
		$$\bb(x,i) \,=\,\left\{ 
		\begin{array}{ll} 
		b\,, & i=0\,, \\ 
		-{\rm sgn}(x)|x|^q\,, & i=1\,,
		\end{array} \right.$$	with $b\in\R$ and $q>1$, and	let 
		$\upsigma(x,i)\equiv0$. 
		The processes 
		\begin{align*}
		\D \X^{(0)}(x;t)&\,=\,\bb\bigl(\X^{(0)}(x;t),0\bigr)\D t+\upsigma\bigl(\X^{(0)}(x;t),0\bigr)\D \B(t)\,=\,b\,\D t\\  \X^{(0)}(x;0)&\,=\,x\in\R\,,
		\end{align*} and
		\begin{align*}
		\D \X^{(1)}(x;t)&\,=\,\bb\bigl(\X^{(1)}(x;t),1\bigr)\D t+\upsigma\bigl(\X^{(1)}(x;t),1\bigr)\D \B(t)\\&\,=\,-{\rm sgn}\bigl(\X^{(1)}(x;t)\bigr)|\X^{(1)}(x;t)|^q\,\D t\\  \X^{(1)}(x;0)&\,=\,x\in\R\,,
		\end{align*}
		are given by $\X^{(0)}(x;t)=x+bt$ and
		$$\X^{(1)}(x;t) \,=\,\left\{ 
		\begin{array}{ll} 
		{\rm sgn}(x)\bigl(|x|^{1-q}+(q-1)t\bigr)^{1/(1-q)}\,, & x\neq 0\,, \\ 
		0\,, & x=0\,.
		\end{array} \right.$$
		Clearly, both $\{\X^{(0)}(x;t)\}_{t\ge0}$ and $\{\X^{(1)}(x;t)\}_{t\ge0}$ 
		are not irreducible and aperiodic. Hence, we cannot apply \Cref{tm:TV} to these processes. In the case when $b\neq0$ the process
		$\{\X^{(0)}(x;t)\}_{t\ge0}$ does not admit an invariant probability measure, while in the case  when $b=0$ it admits uncountably many invariant probability measures: $\{\updelta_{\{x\}}\}_{x\in\R}$. On the other hand, 
		 $\updelta_{\{0\}}$ is a unique invariant probability measure for  $\{\X^{(1)}(x;t)\}_{t\ge0}$. However, convergence of the corresponding semigroup  to  $\updelta_{\{0\}}$ (with respect to some distance function) cannot have exponential rate and this convergence cannot hold in the total variation norm. 
	 
	 Let now $\mathrm{q}_{01}=\mathrm{q}_{10}=1$. Hence, $\uplambda=(1/2,1/2)$.
	  The process $\{(\X(x,i;t),\Lambda(i;t))\}_{t\ge0}$ is also not irreducible and aperiodic (hence, we cannot apply \Cref{tm:TV}), and  since  $	\lVert \updelta_{(x,i)}\mathcal{P}_t
	 -\updelta_{(y,i)}\mathcal{P}_t\rVert_{{\rm TV}} =1$ for all $i\in\mathbb{S}$, $x\neq y$ and $t\ge0$, the semigroup cannot converge to the corresponding invariant probability measure (if it exists) in the total variation norm. 
	 The previous discussion suggest that this convergence (with respect to some distance function) cannot have exponential rate.
	 Observe that in the case when $b=0$
	 the unique invariant probability measure for $\{(\X(x,i;t),\Lambda(i;t))\}_{t\ge0}$ is $\updelta_{\{0\}}\times\uplambda$.

		Let 
		$f(u)=u$ for all $u$ small enough and $f(u)=1-1/(1+u)$ for all $u$ large enough, and  let  $\psi(u)=u^q$ (with $q>1$). 
Obviously, $\bb(x,0)$ satisfies \cref{eq:WASS-fin-bdd1} with $\Gamma_0=0$, and an elementary computation shows that $\bb(x,1)$ satisfies \cref{eq:WASS-fin-bdd1} with some  $\Gamma_1<0$. Hence, we can apply \Cref{tm:WASS-subgeom} to 	$\{(\X(x,i;t),\Lambda(i;t))\}_{t\ge0}$. Further, take $\mathsf{V}(x)=x^2$ and observe that
$$\mathcal{L}_0\mathsf{V}(x)\,=\,2bx\qquad\text{and}\qquad\mathcal{L}_1\mathsf{V}(x)\,=\,-2|x|^{q+1}\,.$$ Thus, for arbitrary small $c_0>0$ and arbitrary large $-c_1>0$ (recall that $q>1$) it holds that  
$$\mathcal{L}_0\mathsf{V}(x)\,\le\,c_0\mathsf{V}(x)\qquad\text{and}\qquad\mathcal{L}_1\mathsf{V}(x)\,\le\,c_1\mathsf{V}(x)$$ for all $|x|$ large enough.
Hence, according to \Cref{tm:WASS-subgeom2,tm:WASS-subgeom3} the process\linebreak $\{(\X(x,i;t),\Lambda(i;t))\}_{t\ge0}$ admits a unique invariant probability measure $\uppi$ and the corresponding semigroup converges to $\uppi$ with respect to $\mathcal{W}_{f,p}$ with subgeometric rate $t^{1/(q-1)}$.
\item[(ii)] Let $\bb(x,i)$, $\{\Lambda(i;t)\}_{t\ge0}$ and $\psi(u)$ be as in (i). Further, let $\upsigma(x,i)\equiv\upsigma(i)$, $\eta\in(0,1)$ and $f(u)=u$. Observe that 
\begin{align*}
\D \X^{(0)}(x;t)&\,=\,b\,\D t+\upsigma(0)\D\B(t)\\  \X^{(0)}(x;0)&\,=\,x\in\R
\end{align*} 
is transient if $b\neq0$ (as a deterministic drift process or Brownian motion with drift) and nullrecurrent if $b=0$ (as a trivial process or Brownian motion). In \cite[Example 3.3]{Lazic-Sandric-2021} it has been shown that \begin{align*}
\D \X^{(1)}(x;t)&\,=\,-{\rm sgn}\bigl(\X^{(1)}(x;t)\bigr)|\X^{(1)}(x;t)|^q\,\D t+\upsigma(1)\D\B(t)\\  \X^{(1)}(x;0)&\,=\,x\in\R
\end{align*}
is subgeometrically ergodic with respect to $\mathcal{W}_{f,1}$ with rate $t^{1/(q-1)}$. Further, obviously $\bb(x,0)$ satisfies \cref{eq:WASS-fin-unbdd1} with $\Gamma_0=0$, and an elementary computation shows that $\bb(x,1)$ satisfies \cref{eq:WASS-fin-unbdd1} with some  $\Gamma_1<0$ for all $x,y\in\R$ satisfying $f(|x-y|)=|x-y|\le\eta$. Hence, 	$\{(\X(x,i;t),\Lambda(i;t))\}_{t\ge0}$ satisfies assumptions of \Cref{tm:WASS-subgeom1.1}. Finally, by completely the same reasoning as in (i) we again conclude that  $\{(\X(x,i;t),\Lambda(i;t))\}_{t\ge0}$ is subgeometrically ergodic with respect to $\mathcal{W}_{f,1}$ with  rate $t^{1/(q-1)}$.
\end{itemize}		}
\end{example}

\section{Ergodicity of some regime-switching Markov processes with jumps}\label{S5}

In this section, we briefly discuss ergodicity properties of a class of regime-switching Markov processes with jumps.

One of the most common approaches in obtaining  Markov processes with jumps (from a given Markov process) is through the Bochner's subordination method. 
Among the most interesting examples are the rotationally invariant stable L\'evy processes, which can be viewed as subordinate Brownian motions. 
Recall, a subordinator  $\{\sub(t)\}_{t\ge0}$ is a non-decreasing L\'{e}vy process on $\left[0,\infty\right)$ with Laplace transform
$$\mathbb{E}\left[\E^{-u\sub(t)}\right] \,=\, \E^{-t\phi(u)}\,.$$
The characteristic (Laplace) exponent $\phi:(0,\infty)\to(0,\infty)$  is a Bernstein function, that is, it is of class $\mathcal{C}^\infty$ and $(-1)^n\phi^{(n)}(u)\ge0$ for all $n\in\N$. It is well known that every Bernstein function  admits a unique (L\'{e}vy-Khintchine) representation 
$$\phi(u)\,=\,\beta u+\int_{(0,\infty)}\bigl(1-\E^{-uy}\bigr)\,\upnu(\D y)\,, $$
where $\beta\geq0$ is the drift parameter and $\upnu$ is a L\'{e}vy measure, that is, a measure on $\mathfrak{B}((0,\infty))$ satisfying $\int_{(0,\infty)}(1\wedge y)\upnu(\D y)<\infty$.
For more on subordinators and Bernstein functions we refer the readers 	to the monograph \cite{Schilling-Song-Vondracek-Book-2012}.
Let now  $\{\sub(t)\}_{t\ge0}$ be a subordinator with characteristic exponent $\phi(u)$, independent of   $\{(\X(x,i;t),\Lambda(x,i;t))\}_{t\ge0}$. 
The process $(\X^\phi(x,i;t),\Lambda^\phi(x,i;t))\df (\X(x,i;\sub(t)),\Lambda(x,i;\sub(t)))$, $t\ge0$, obtained from $\{(\X(x,i;t),\Lambda(x,i;t))\}_{t\ge0}$ by 
a random time change  through $\{\sub(t)\}_{t\ge0}$, is referred to as the subordinate process $\{(\X(x,i;t),\Lambda(x,i;t))\}_{t\ge0}$ with subordinator $\{\sub(t)\}_{t\ge0}$ in the sense of Bochner. It is known that many fine properties of Markov processes (and the corresponding semigroups) are preserved under subordination. It is easy to see that $\{(\X^\phi(x,i;t),\Lambda^\phi(x,i;t))\}_{t\ge0}$ is again a Markov process with  transition kernel
$$\p^\phi\bigl(t,(x,i),\D y\times\{j\}\bigr)\,=\,\int_{\left[0,\infty\right)} \p\bigl(s,(x,i),\D y\times\{j\}\bigr)\,\upmu_t(\D s)\,,$$
where $\upmu_t(\cdot)=\mathbb{P}(\sub(t)\in\cdot)$ is the transition probability of $\sub(t)$, $t\ge0$. Also, it is elementary to check that if $\uppi$ is an invariant probability measure for $\{(\X(x,i;t),\Lambda(x,i;t))\}_{t\ge0}$, then it is also invariant for the subordinate process $\{(\X^\phi(x,i;t),\Lambda^\phi(x,i;t))\}_{t\ge0}$. In \cite{Deng-Schilling-Song-2017} and \cite[Proposition 3.7]{Lazic-Sandric-2021} it has been shown that if  $\{(\X(x,i;t),\Lambda(x,i;t))\}_{t\ge0}$ is sub-geometrically ergodic with Borel measurable rate $r(t)$ (with respect to the total variation distance or an $\mathcal{L}^p$-Wasserstein distance), then $\{(\X^\phi(x,i;t),\Lambda^\phi(x,i;t))\}_{t\ge0}$ is subgeometrically ergodic with rate $r_\phi(t)=\mathbb{E}[r(\sub(t))]$ (in the total variation distance case) and $r_\phi(t)=(\mathbb{E}[r^p(\sub(t))])^{1/p}$  (in the $\mathcal{L}^p$-Wasserstein distance case).  Therefore, as an direct application of \Cref{tm:TV,tm:WASS-subgeom2} we obtain subgeometric ergodicity results  for a class of subordinate regime-switching diffusion processes.

In the end, we remark that \Cref{tm:WASS-subgeom,tm:WASS-subgeom1.1,tm:WASS-subgeom2,tm:WASS-subgeom3} can be stated in a slightly more general form by replacing the Brownian motion $\{\B(t)\}_{t\ge0}$ in \cref{eq1} by a general L\'evy process.
Let $\{\LL(t)\}_{t\ge0}$  be an $n$-dimensional L\'evy process (starting from the origin) with L\'evy triplet $(\beta,\gamma,\upnu)$. Consider the regime-switching jump diffusion process $\{(\X(x,i;t),\Lambda(i;t))\}_{t\ge0}$ with the first component given by
\begin{equation}
\begin{aligned}
\label{eq3}
\D \X(x,i;t) &\,=\, \bb\bigl(\X(x,i;t),\Lambda(i;t)\bigr)\D t+\upsigma\bigl(\Lambda(i;t-)\bigr)\D  \LL(t)\\ \X(x,i;0)&\,=\, x \in\R^d\\
\Lambda(i;0)&\,=\, i\in\mathbb{S}\,,
\end{aligned}
\end{equation}   and  the second  component, as before, is a right-continuous temporally-homogeneous Markov chain  with finite state space $\mathbb{S}$.
The processes  $\{\LL(t)\}_{t\ge0}$ and $\{\Lambda(i;t)\}_{t\ge0}$ are independent and defined on a stochastic basis $(\Omega, \mathcal{F}, \{\mathcal{F}_t\}_{t\ge0},\Prob)$ (satisfying the usual conditions). 
Assume that the coefficients $\bb:\R^d\times\mathbb{S}\to\R^d$ and $\upsigma:\mathbb{S}\to\R^{d\times n}$, and the process $\{\Lambda(i;t)\}_{t\ge0}$ satisfy the following:

\medskip

\begin{description}
	\item[($\widetilde{\textbf{A1}}$)]   for any $r>0$ and $i\in\mathbb{S}$, $$\sup_{x\in\mathscr{B}_r(0)}|\bb(x,i)|\,<\,\infty$$

	\medskip

	\item[($\widetilde{\textbf{A2}}$)] for each $(x,i)\in\R^d\times\mathbb{S}$ the RSSDE in  \cref{eq3} admits a unique nonexplosive strong solution $\{X(x,i;t)\}_{t\ge0}$ which has c\`{a}dl\`{a}g sample paths

	\medskip
	
	\item[($\widetilde{\textbf{A3}}$)]	the process $\{(\mathsf{X}(x,i;t),\Lambda(i;t))\}_{t\ge0}$ is a temporally-homogeneous strong Markov process with transition kernel $\p(t,(x,i),\D y\times \{j\})=\Prob((\X(x,i;t),\Lambda(i;t))\in\D y\times \{j\})$
	
	\medskip
	
	\item[($\widetilde{\textbf{A4}}$)]	the corresponding semigroup of linear operators $\{\mathcal{P}_t\}_{t\ge0}$ satisfies the 
	$\mathcal{C}_b$-Feller property
	
	\medskip
	
	\item[($\widetilde{\textbf{A5}}$)]	for any $(x,i)\in\R^d\times\mathbb{S}$ and $f\in \mathcal{C}^2(\R^d\times\mathbb{S})$ such that $(x,i)\mapsto\int_{\R^d}f(x+y,i)\upnu_i(\D y)$ is locally bounded, the process $$\left\{f\bigl(\X(x,i;t),\Lambda(x,i;t)\bigr)-f(x,i)-\int_0^t\mathcal{L}f\bigl(\X(x,i;s),\Lambda(x,i;s)\bigr)\D s\right\}_{t\ge0}$$ is a $\mathbb{P}$-local martingale, where $\upnu_i(B)=\upnu(\{x\in\R^n\colon \upsigma(i) x\in B\})$ for $B\in\mathfrak{B}(\R^d)$ and 
	$$\mathcal{L}f(x,i)\,=\,\mathcal{L}_if(x,i)+\mathcal{Q}f(x,i)$$ with
\begin{align*}\mathcal{L}_if(x)\,=\,&\left\langle \bb(x,i)+\upsigma(i)\beta+\int_{\R^n}\upsigma(i)y\bigl(\mathbb{1}_{\mathscr{B}_1(0)}(\upsigma(i)y)-\mathbb{1}_{\mathscr{B}_1(0)}(y)\bigr)\upnu(\D y),\nabla f(x)\right \rangle\\&+\frac{1}{2}{\rm Tr}\bigl(\upsigma(i)\gamma\upsigma(i)^{T}\nabla^2f(x)\bigr)+\int_{\R^d}\bigl(f(x+y)-f(x)-\langle y,\nabla f(x)\rangle\mathbb{1}_{\mathscr{B}_1(0)}(y)\bigr)\upnu_i(\D y)\end{align*} 
and $\mathcal{Q}=(\mathrm{q}_{ij})_{i,j\in\mathbb{S}}$ being the infinitesimal generator of the process $\{\Lambda(i;t)\}_{t\ge0}$. 
\end{description}

\medskip

\noindent	
We refer the readers to    \cite{Xi-Yin-Zhu-2019} (see also \cite{Kunwai-Zhu-2020})   for conditions  ensuring   \textbf{($\widetilde{\textbf{A1}}$)}-\textbf{($\widetilde{\textbf{A5}}$)}. It is straightforward to check that \Cref{tm:WASS-subgeom} (and \Cref{tm:WASS-subgeom2,tm:WASS-subgeom3}) holds also in this situation (under the additional assumption that  the functions $\mathcal{V}(x,i)$, $\mathsf{V}(x)$ and $\theta(u)$ appearing in  \Cref{tm:WASS-subgeom2,tm:WASS-subgeom3}   are such that  $(x,i)\mapsto\int_{\R^d}\mathcal{V}(x+y,i)\upnu_i(\D y)$, $(x,i)\mapsto\int_{\R^d}\mathsf{V}(x+y)\upnu_i(\D y)$ and $(x,i)\mapsto\int_{\R^d}\theta\circ\mathsf{V}(x+y)\upnu_i(\D y)$ are locally bounded). On the other hand, in order to conclude the results from \Cref{tm:WASS-subgeom1.1}  we need to extend the results from \Cref{lemma:X_square} to the jump case. More specifically, \Cref{tm:WASS-subgeom1.1}  follows by replacing \cref{linear} by \cref{lin_jump} and $\int_{\R^n}(|y|^2\vee |y|^4)\upnu(\D y)<\infty$.

\begin{lemma}\label{lm:X^2}
	
	Assume that $\int_{\R^n}(|y|^2\vee |y|^4)\upnu(\D y)<\infty$ (or, equivalentely, $\mathbb{E}[|\LL_t|^4]<\infty$ for all $t\ge0$) and
	\begin{equation}
	\begin{aligned}\label{lin_jump} &2\bigl\langle x, \mathrm{b} (x,i)+\upsigma(i)\beta +\int_{\R^n}\upsigma(i)y\bigl(\mathbb{1}_{\R^n}\bigl(\upsigma(i)y\bigr)-\mathbb{1}_{\mathscr{B}_1(0)}(y)\bigr)\upnu(\D y)\bigl\rangle\\&  +{\rm Tr}\bigl(\upsigma(i)\gamma\upsigma(i)^{T}\bigr) +\int_{\R^d}|y|^2\upnu_{i}(\D y)\\&\,\leq\, K (1+ |x|^2)\,.
	\end{aligned}\end{equation}
		Then, $$ \mathbb{E} \bigl[ |\X(x,i;t)|^2\bigr] \,\leq\  \bigl(1 +|x|^2\bigr) \E^{Kt}\,.$$
	Furthermore,  for any $\{\mathcal{F}_t\}_{t\ge0}$-stopping time $\tau$ such that $\mathbb{E}\bigl[\E^{2K\tau}\bigr]<\infty$ it follows that
	\begin{align*} &\mathbb{E} \left[ |\X(x,i;\tau)|^2 \right]\\ &\,\le\, |x|^2 +  \bigl(1+|x|^2\bigr)\mathbb{E}\left[\E^{K\tau}\right]\\
	& \ \ \ \ \ +\bigl(1+|x|^2\bigr)\mathbb{E}\left[  \E^{2K  \tau} \right]^{1/2}	 \Bigg(4\sup_{j\in\mathbb{S}}\mathrm{Tr}\bigl(\upsigma(j)\upsigma(j)^T\bigr) +\sup_{j\in\mathbb{S}}\mathrm{Tr}\bigl(\upsigma(j)\upsigma(j)^T\bigr)^2\int_{\R^n}|y|^4\nu(\D y)/K\\
	&\ \ \ \ \ +4\sup_{j\in\mathbb{S}}\mathrm{Tr}\bigl(\upsigma(j)\upsigma(j)^T\bigr)\int_{\R^n}|y|^2\upnu(\D y)  +4\sup_{j\in\mathbb{S}}\mathrm{Tr}\bigl(\upsigma(j)\upsigma(j)^T\bigr)^{3/2}\int_{\R^n}|y|^3\upnu(\D y)\Bigg)^{1/2}\,.\end{align*}
\end{lemma}
\begin{proof}
	For $n\in\N$, let $f_n:\R^d\to[0,\infty)$ be such that  $f_n\in\mathcal{C}^2_b(\R^d)$ (the space of bounded and twice continuously differentiable functions with bounded first and second order derivatives), $f_n(x)=|x|^2$ on $\mathscr{B}_{n+1}(0)$ and $f_n(x)\le f_{n+1}(x)$ for all $x\in\R^d$, and
	$$\tau_n\,\df\, \inf \{ t\geq 0 : |\X(x,i;t)| \geq n \}\,.$$ Further, for $t>0$ and $B\in\mathfrak{B}(\R^n)$ denote $$\mathrm{N}((0,t],B)\,\df\, \sum_{0<s\le t}\mathbb{1}_B(\LL(s)-\LL(s-))\qquad\text{and}\qquad\tilde{\mathrm{N}}(\D t,\D y)\,\df\, \mathrm{N}(\D s,\D y)-\upnu(\D y)\D s\,.$$
	By employing It\^{o}'s formula and  the assumption that $\int_{\R^n}(|y|^2\vee |y|^4)\upnu(\D y)<\infty$ we conclude that for $n$ large enough,
	\begin{align*}
	&f_n\bigl(\X(x,i;t  \wedge \tau_n)\bigr) \\&\,=\, f_n(x) +  \int_0^{t  \wedge \tau_n} \Bigg(\Big\langle \nabla f_n\bigl(\X(x,i;s)\bigr), \bb\bigl(\X(x,i;s), \Lambda(i;s)\bigr) +\upsigma\bigl(\Lambda(i;s)\bigr)\beta\\&\hspace{3.5cm}+\int_{\R^n}\upsigma\bigl(\Lambda(i;s)\bigr)y\bigl(\mathbb{1}_{\R^n}\bigl(\upsigma\bigl(\Lambda(i;s)\bigr)y\bigr)-\mathbb{1}_{\mathscr{B}_1(0)}(y)\bigr)\upnu(\D y)\Big\rangle \\&\ \ \ \ \  +\frac{1}{2}{\rm Tr}\bigl(\upsigma\bigl(\Lambda(i;s)\bigr)\gamma\upsigma\bigl(\Lambda(i;s)\bigr)^{T}\nabla^2f_n\bigl(\X(x,i;s)\bigr)\bigr)\\& \ \ \ \ \ +\int_{\R^d}\left(f_n\bigl(\X(x,i;s)+y\bigr)-f_n\bigl(\X(x,i;s)\bigr)-\bigl\langle y,\nabla f_n\bigl(\X(x,i;s)\bigr)\bigr\rangle\right)\upnu_{\Lambda(i;s)}(\D y)\Bigg)\D s  \\
	&\ \ \ \ \   + \int_0^{t  \wedge \tau_n}  \nabla f_n\bigl(\X(x,i;s)\bigr)^T \upsigma\bigl(\Lambda(i;s)\bigr) \D\B(s) \\
	&\ \ \ \ \ +\int_0^{t  \wedge \tau_n}\int_{\R^n} \bigr( f_n\bigl(\X(x,i;s-)+\upsigma\bigl(\Lambda(i;s-)\bigr)y\bigr)-f_n\bigl(\X(x,i;s-)\bigr)\bigl) \tilde {\mathrm{N}}(\D y,\D s) \\
	&\,=\, 
	|x|^2 +  2\int_0^{t  \wedge \tau_n} \Bigg(\Big\langle \X(x,i;s), \bb\bigl(\X(x,i;s), \Lambda(i;s)\bigr) +\upsigma\bigl(\Lambda(i;s)\bigr)\beta\\&\hspace{3.5cm}+\int_{\R^n}\upsigma\bigl(\Lambda(i;s)\bigr)y\bigl(\mathbb{1}_{\R^n}\bigl(\upsigma\bigl(\Lambda(i;s)\bigr)y\bigr)-\mathbb{1}_{\mathscr{B}_1(0)}(y)\bigr)\upnu(\D y)\Big\rangle \\&\ \ \ \ \  +{\rm Tr}\bigl(\upsigma\bigl(\Lambda(i;s)\bigr)\gamma\upsigma\bigl(\Lambda(i;s)\bigr)^{T}\bigr) +\int_{\R^d}|y|^2\upnu_{\Lambda(i;s)}(\D y)\Bigg)\D s  \\
	&\ \ \ \ \   + \int_0^{t  \wedge \tau_n}  \nabla f_n\bigl(\X(x,i;s)\bigr)^T \upsigma\bigl(\Lambda(i;s)\bigr) \D\B(s) \\
	&\ \ \ \ \ +\int_0^{t  \wedge \tau_n}\int_{\R^n} \bigr( f_n\bigl(\X(x,i;s-)+\upsigma\bigl(\Lambda(i;s-)\bigr)y\bigr)-f_n\bigl(\X(x,i;s-)\bigr)\bigl) \tilde {\mathrm{N}}(\D y,\D s) \\ 
&\,\leq\,
|x|^2 + K\int_0^t \bigl( 1 + |\X(x,i;s)|^2\mathbb{1}_{[0,\tau_n]}(s)\bigr) \, \D s \\  &\ \ \ \ \   + \int_0^{t  \wedge \tau_n}  \nabla f_n\bigl(\X(x,i;s)\bigr)^T \upsigma\bigl(\Lambda(i;s)\bigr) \D\B(s) \\
&\ \ \ \ \ +\int_0^{t  \wedge \tau_n}\int_{\R^n} \bigr( f_n\bigl(\X(x,i;s-)+\upsigma\bigl(\Lambda(i;s-)\bigr)y\bigr)-f_n\bigl(\X(x,i;s-)\bigr)\bigl) \tilde {\mathrm{N}}(\D y,\D s)\,,
	\end{align*} where in the last step we used \cref{lin_jump}.
	By taking expectation, we have that 
	\begin{align*}&1+ \mathbb{E} \left[ |\X(x,i;t )|^2 \mathbb{1}_{[0,\tau_n)}(t)\right]\\&\,\le\, 1+\mathbb{E} \left[ f_n\bigl(\X(x,i;t  \wedge \tau_n)\bigr) \right]\\&\, \leq\,  1 + |x|^2 + K\int_0^t \bigl( 1 + \mathbb{E} \left[ |\X(x,i;s)|^2 \right]\mathbb{1}_{[0,\tau_n]}(s) \bigr) \D s\\
	&\, =\,  1 + |x|^2 + K\int_0^t \bigl( 1 + \mathbb{E} \left[ |\X(x,i;s )|^2 \mathbb{1}_{[0,\tau_n)}(s)\right] \bigr) \D s\,.\end{align*}
	The first assertion now follows by employing Gr\" onwall's inequality and Fatou's lemma.

	Let now  $\tau$ be a stopping time such that $\mathbb{E}[\E^{2K\tau}]<\infty$. It\^{o}'s lemma then gives
	\begin{align*}
	|\X(x,i;t)|^2
	\,\leq\,& |x|^2 + Kt + K\int_0^t |\X(x,i;s)|^2 \D s+2 \int_0^t \X(x,i;s)^T \upsigma\bigl(\Lambda(i;s)\bigr) \D \B(s)\\
	&+\int_0^{t  }\int_{\R^n} \bigr( |\upsigma\bigl(\Lambda(i;s-)\bigr)y|^2+2 \X(x,i;s-)^T\upsigma\bigl(\Lambda(i;s-)\bigr)y\bigl) \tilde {\mathrm{N}}(\D y,\D s)\,. 
	\end{align*}
	Denote \begin{align*}\alpha(t)\,\df\,&2 \int_0^t \X(x,i;s)^T \upsigma\bigl(\Lambda(i;s)\bigr) \D \B(s)\\
	&+\int_0^{t  }\int_{\R^n} \bigr( |\upsigma\bigl(\Lambda(i;s-)\bigr)y|^2+2 \X(x,i;s-)^T\upsigma\bigl(\Lambda(i;s-)\bigr)y\bigl) \tilde {\mathrm{N}}(\D y,\D s)\,.\end{align*}
	Gr\" onwall's inequality then gives
	\begin{align*} |\X(x,i;t)|^2 &\,\leq\,  |x|^2 + Kt + \alpha(t) + \int_0^t K \left( |x|^2 + Ks + \alpha(s) \right) \E^{K(t-s)} \D s\,.
	\end{align*}
	Consequently,  
	\begin{align*} &|\X(x,i;t \wedge \tau)|^2\\&\,\leq\, |x|^2 + K(t \wedge \tau) + \alpha(t\wedge\tau)  + |x|^2\bigl(\E^{Kt \wedge \tau}-1\bigr)+\E^{Kt \wedge \tau}-K(t \wedge \tau )\\&\ \ \ \ \ +K
	\int_0^{t}  \left( \mathbbm{1}_{[0, \tau ]}(s)\, \E^{K(t \wedge \tau -s)} \alpha(s)\right)  \D s\\
	&\,\le\, |x|^2 + \alpha(t\wedge\tau) + \bigl(1+|x|^2\bigr)\E^{K\tau}+K
	\int_0^{t}  \left( \mathbbm{1}_{[0, \tau ]}(s)\, \E^{K(t \wedge \tau -s)}\alpha(s)\right)  \D s\,.
	\end{align*}
	Taking expectation we have that
	\begin{align*} &\mathbb{E}\bigl[|\X(x,i;t \wedge \tau)|^2\bigr]\\&\,\leq\,
	|x|^2 +  \bigl(1+|x|^2\bigr)\mathbb{E}\left[\E^{K\tau}\right] +K
	\int_0^{t}  \mathbb{E}\left[ \mathbbm{1}_{[0, \tau]}(s)\, \E^{K(t \wedge \tau -s)}\alpha(s)\right]  \D s
	\\&\,\leq\,
	|x|^2 +  \bigl(1+|x|^2\bigr)\mathbb{E}\left[\E^{K\tau}\right]
	+K \int_0^t \mathbb{E}\left[ \mathbbm{1}_{[0,  \tau ]}(s)\, \E^{2K (t \wedge \tau -s)}  \right]^{1/2}  \mathbb{E}\left[ \alpha(s)^2  \right]^{1/2}\D s \\
	\\&\,\leq\,
	|x|^2 +  \bigl(1+|x|^2\bigr)\mathbb{E}\left[\E^{K\tau}\right]
	\\&\ \ \ \ \ + K \mathbb{E}\left[  \E^{2K  \tau} \right]^{1/2}	\int_0^t \E^{-Ks}  \Bigg(4\sup_{j\in\mathbb{S}}\mathrm{Tr}\bigl(\upsigma(j)\upsigma(j)^T\bigr)\mathbb{E}\left[ \int_0^s  |\X(x,i;u)|^2 \D u \right]\\
	&\ \ \ \ \ +\sup_{j\in\mathbb{S}}\mathrm{Tr}\bigl(\upsigma(j)\upsigma(j)^T\bigr)^2s\int_{\R^n}|y|^4\nu(\D y)\\
	&\ \ \ \ \ +4\sup_{j\in\mathbb{S}}\mathrm{Tr}\bigl(\upsigma(j)\upsigma(j)^T\bigr)\int_{\R^n}|y|^2\upnu(\D y)\,\mathbb{E}\left[\int_0^s|\X(x,i;u)|^2 \D u\right]\\& \ \ \ \ \ +4\sup_{j\in\mathbb{S}}\mathrm{Tr}\bigl(\upsigma(j)\upsigma(j)^T\bigr)^{3/2}\int_{\R^n}|y|^3\upnu(\D y)\,\mathbb{E}\left[\int_0^s|\X(x,i;u)| \D u\right]\Bigg)^{1/2} \D s \\
	&\,\leq\,|x|^2 +  \bigl(1+|x|^2\bigr)\mathbb{E}\left[\E^{K\tau}\right]\\&\ \ \ \ \
	+ K \mathbb{E}\left[  \E^{2K  \tau} \right]^{1/2}	\int_0^\infty \E^{-Ks}  \Bigg(4\sup_{j\in\mathbb{S}}\mathrm{Tr}\bigl(\upsigma(j)\upsigma(j)^T\bigr)\bigr(1+|x|\bigl)^2\E^{Ks}\\
	&\ \ \ \ \  +\sup_{j\in\mathbb{S}}\mathrm{Tr}\bigl(\upsigma(j)\upsigma(j)^T\bigr)^2s\int_{\R^n}|y|^4\nu(\D y)\\
	&\ \ \ \ \ +4\sup_{j\in\mathbb{S}}\mathrm{Tr}\bigl(\upsigma(j)\upsigma(j)^T\bigr)\int_{\R^n}|y|^2\upnu(\D y)\bigr(1+|x|\bigl)^2\E^{Ks} \\& \ \ \ \ \ 4\sup_{j\in\mathbb{S}}\mathrm{Tr}\bigl(\upsigma(j)\upsigma(j)^T\bigr)^{3/2}\int_{\R^n}|y|^3\upnu(\D y)\bigr(1+|x|\bigl)\E^{Ks/2}\Bigg)^{1/2} \D s\\
	&\,\leq\,|x|^2 +  \bigl(1+|x|^2\bigr)\mathbb{E}\left[\E^{K\tau}\right]\\
	& \ \ \ \ \ +\mathbb{E}\left[  \E^{2K  \tau} \right]^{1/2}	 \Bigg(4\sup_{j\in\mathbb{S}}\mathrm{Tr}\bigl(\upsigma(j)\upsigma(j)^T\bigr)\bigr(1+|x|\bigl)^2  +\sup_{j\in\mathbb{S}}\mathrm{Tr}\bigl(\upsigma(j)\upsigma(j)^T\bigr)^2\int_{\R^n}|y|^4\nu(\D y)/ K\\
	&\ \ \ \ \ +4\sup_{j\in\mathbb{S}}\mathrm{Tr}\bigl(\upsigma(j)\upsigma(j)^T\bigr)\int_{\R^n}|y|^2\upnu(\D y)\bigr(1+|x|\bigl)^2 \\ &\ \ \ \ +4\sup_{j\in\mathbb{S}}\mathrm{Tr}\bigl(\upsigma(j)\upsigma(j)^T\bigr)^{3/2}\int_{\R^n}|y|^3\upnu(\D y)\bigr(1+|x|\bigl)\Bigg)^{1/2}\,,
	\end{align*}
	where in the third step we used It\^{o}'s isometry and in the fourth step we used the first assertion of the lemma.
\end{proof}




\section*{Acknowledgements}
 Financial support through the \textit{Croatian Science Foundation} under project 8958 (for P. Lazi\'c), and  \textit{Alexander-von-Humboldt Foundation} under project No. HRV 1151902 HFST-E and \textit{Croatian Science Foundation} under project 8958  (for N. Sandri\'c)
are  gratefully acknowledged. 
We also thank the anonymous referees for the helpful comments that have led to significant improvements of the results in the article.

\bibliographystyle{abbrv}
\bibliography{References}

\end{document}